\documentclass[preprint,12pt]{elsarticle}
\usepackage[top=3cm,bottom=3cm,right=2cm,left=2cm]{geometry}
\usepackage{amsfonts,amsmath}
\usepackage{caption}
\usepackage{tabulary}
\usepackage{booktabs}

\usepackage{blindtext}
\usepackage{setspace}

        \newtheorem{theorem}{Theorem}
 
        \newtheorem{lemma}[theorem]{Lemma}
        \newtheorem{corollary}[theorem]{Corollary}
        \newtheorem{remark}[theorem]{Remark}
        \newtheorem{assumption}[theorem]{Assumption}
\usepackage{rotating}

\newcommand{\sn}{ |\kern-0.25ex|\kern-0.25ex| }
\renewcommand{\div}{\mbox{\normalfont div}}
\renewcommand{\sin}{\mbox{\normalfont sin}}
\renewcommand{\cos}{\mbox{\normalfont cos}}
\usepackage{verbatim}
\usepackage{epsfig}

\usepackage{amssymb}

\newcounter{bla}

\journal{Computer Physics Communications}

\begin{document}


\title{A convergence analysis of Generalized Multiscale Finite Element Methods}

\author{\textbf{Eduardo Abreu}$^{1}$, \textbf{Ciro Diaz}$^{1}$}

\author{\textbf{Juan Galvis}$^2$ \corref{cor1}}
\cortext[cor1]{Email address : jcgalvisa@unal.edu.co}

\address{$^{1}$
  Department of Applied Mathematics (IMECC) \\
  University of Campinas (UNICAMP) \\
  Campinas 13.083-970, SP, Brazil}

\address{$^{2}$
  Departamento de Matem\'aticas, Universidad Nacional de Colombia,\\
  Carrera 45 No 26-85 - Edificio Uriel Gutierr\'ez, Bogot\'a D.C. -
  Colombia}

\begin{abstract}
In this paper, we consider an approximation method, and a novel
general analysis, for second-order elliptic differential equations with
heterogeneous multiscale coefficients. We obtain convergence of the
Generalized Multi-scale Finite Element Method (GMsFEM) method that uses
local eigenvectors in its construction. The analysis presented here
can be extended, without great difficulty, to more sophisticated GMsFEMs.
For concreteness, the obtained error estimates generalize and simplify the convergence analysis   of [J. Comput. Phys. 230 (2011), 937-955]. The
GMsFEM method construct basis functions that are obtained by
multiplication of (approximation of) local eigenvectors by partition of
unity functions. Only important eigenvectors are used in the construction. 
The error estimates are general
and are written in terms of the eigenvalues of the eigenvectors not used
in the construction. The error analysis involve local and global
norms that measure the decay of the expansion of the solution in terms of
local eigenvectors. Numerical experiments are carried out to verify the
feasibility of the approach with respect to the convergence and stability
properties of the analysis in view of the good scientific computing
practice.
\end{abstract}


\maketitle

\date{}

\def \bn {\hfill \\ \smallskip \noindent}
\def \proof{\bn {\bf Proof.}}
\def \Box {\vrule height5pt width5pt depth0pt}
\def \endproof{\hfill \Box \vskip .5cm}

\pagestyle{myheadings}
\thispagestyle{plain}
\markboth{hc.tex}{}

\tableofcontents 

\section{Introduction}

Approximation of partial differential equations posed on domains with multiscale and heterogeneous 
properties appear in variety of applications. For instance, when modeling 
subsurface flow scenarios, subsurface properties typically vary several orders of magnitude over multiple scales.  In this case,  the high-contrast  in the properties 
such as permeability raises additional issues  to be consider when constructing
approximation of solutions. 
Several multiscale models
to efficiently  solve flow and transport processes have been considered. In despite of many contributions, the design and mathematical analysis of high-contrast multiscale problems continue being a challenging 
problem; See 
for instance 
\cite{hw97, aarnes, aej07, bo09,  AKL, arbogast02, apwy07, cdgw03, eghe05, jennylt03, cgh09, hughes98,eh09,ge09_1,
ge09_1reduceddim, sarkisguzman, sarkisburman, tat}.
These approaches approximate the effects of the fine-scale features using 
a coarse mesh. They attempt to capture the fine scale effects on a coarse grid via localized basis functions.  The main idea of Multiscale Finite Element Methods (MsFEMs) is to construct
basis functions that are used
to approximate the solution on a coarse grid. 
The accuracy of MsFEMs is found to be very sensitive
to the particularities of the construction of the  basis functions 
(e.g., boundary conditions of local problems). See for instance \cite{eh09,ehg04, ehw99}).

It is known that the construction of the basis functions
need to be carefully designed in order to obtain accurate
coarse-scale approximations of the solution (e.g., \cite{eh09}).
In particular,  the resulting basis functions need to have similar oscillatory
behavior as the fine-scale solution. In classical multiscale methods, a number of approaches are proposed
to construct  basis functions, 
e.g., oversampling techniques or the use of limited global
information (e.g., \cite{hw97, eh09}) that employs
 solutions in larger regions
to reduce localization errors. 
Recently, a new and promising methodology was introduced for the construction 
of basis function. This methodology is referred as to Generalized Multiscale 
Finite Element Method (GMsFEM). The main goal of GMsFEMs
is to construct 
 coarse spaces for MsFEMs
that result in accurate coarse-scale solutions. 
This methodology was first developed in \cite{ge09_1,ge09_1reduceddim, eglw11}
in connections with the robustness of domains decomposition iterative methods
for solving the  elliptic equation with heterogeneous coefficients subjected to
appropriate boundary conditions
\begin{equation}
-\mbox{div}(\kappa(x)\nabla u)=f,
\label{eq:problem1}
\end{equation}
where $\kappa(x)$ is a heterogeneous scalar field with high-contrast. In particular, 
it is  assumed that $\kappa(x)\geq c_0>0$ (bounded below), while
$\kappa(x)$ can have very large values.

A main ingredient in the construction was the use of local generalize eigenvalue 
problems and (possible multiscale) partition of unity functions to construct the coarse spaces.
Besides using one coarse function per coarse node, in the GMsFEM
it was proposed to use several multiscale basis functions per coarse node.
These basis functions represent
important features of the solution within a coarse-grid 
block and they are 
computed using eigenvectors of an eigenvalue problem.
Then, in  the works \cite{egw10, CEG, EGG_MultiscaleMOR, Review},  some studies 
of the coarse approximation properties of the GMsFEM were carried out. 
In these works and for applications to high-contrast problems, methodologies to keep 
small the dimension of the resulting coarse space were successfully proposed.
The use of coarse spaces that somehow incorporates important modes of a (local) energy 
related to the problem motivated the general version of the GMsFEM.
Thus,a more general and practical GMsFEM 
was then developed in \cite{egh12} where several (more practical) options to 
compute important modes to be include in the coarse space was used. 
See also \cite{egt11} for an earlier construction.
It is important 
to mention that the methodology 
in \cite{egh12} was designed for parametric and nonlinear problems and 
can be applied for variety of applications as it have been shown in 
recent developments not review here.

In this paper, we prove convergence of the GMsFEM method that uses
local eigenvectors as developed in \cite{CEG, ge09_1,ge09_1reduceddim, eglw11}. The analysis presented here can be extended, without great difficulty, to more sophisticated GMsFEMs.
Some convergence analysis of the GMsFEM, using local eigenvectors
or approximation of them was obtained in \cite{egw10}.  The prove, as usual in finite 
element analysis, focuses on constructing interpolation operator to the coarse finite 
element space. The a priory error is obtained for square integrable 
right hand side $f$ in (\ref{eq:problem}).
Additionally, in \cite{egw10}  the authors make some assumptions concerning integrability of
residuals and also concerning boundedness of the quotients of local energy norms 
with weight $\kappa$ and $\kappa \chi^2$ where $\chi^2$ is a especial 
partition of unity function. These assumptions are hard to verify in practice. 
Moreover, in the analysis they use a Caccioppoli inequality to write energy 
estimates from a region to a bigger region. Therefore, extensions of the analysis in \cite{egw10} to other equations and/or 
different discretization is not straightforward.

In this paper, we substantially simplify the analysis of GMsFEM methods and remove the 
assumptions used in \cite{egw10} to obtain convergence, 
yielding a general convergence proof and more suited for computational practice. We 
assume square integrability of the right hand side $f$.  In order to obtain 
error bounds in terms of the decay of the eigenvalues used in the construction we assume that the problem is regular in the sense that 
the solution can be well approximated by local eigenvectors which in the case of smooth coefficients, square integrable right hand side and convex domains,  is implied by the classical regularity of the problem.

It is worth to mention a main difference between the classical finite element analysis 
and the analysis of GMsFEM procedures for the case of heterogeneous multiscale coefficients. In the usual finite element analysis, to write the interpolation error 
estimates, it is assumed that the solution is smooth enough or regular enough 
in the classical (Sobolev) sense. This is done while using Hilbert norms (at least 
for elliptic problems). 
In the case of discontinuous multiscale coefficients, it is well know that solutions 
are not smooth in the classical sense. Then, the classical finite element 
analysis arguments do not work.
In this paper, we are able to write interpolation error estimates using norms 
suitable for the problem at hand. In particular, to measure the ``smoothness'' of the 
solution we use the decay of the expansion of the solution in terms of global 
eigenvectors. This is motivated by the fact that, for a given elliptic operator, 
the eigenvectors are a good model for smooth functions in the scale of norms
generated by powers of the operator. 
We define then global norms, using the decay of the expansion over global 
eigenvectors. We also define local norms using the decay of the expansion 
in terms of local eigenvectors (computed locally in a coarse node neighborhood). 
The main result of this paper is that we can compare the new local and global norms.
With this new norms, we are able to write approximation results for the 
interpolation of functions that solve (\ref{eq:problem}) with square 
integrable right hand side. We also prove error estimates in terms of the eigenvalues 
of the eigenvalue problem used in the construction.

The rest of the paper is organized as follow.
In Section \ref{sec:preliminaries} we present some preliminaries on multiscale methods. 
In Section \ref{sec:global} we collect some facts on the global eigenvalue problem related 
to the problem. Here we introduce a scale of global norms used for the analysis. These 
norms measure the decay of the expansion in terms of global eigenvectors.
In Sections \ref{sec:localN} and \ref{sec:localD} we study the local eigenvalue problems 
also using norms that 
measure the  decay of the expansion in terms of the local eigenvectors. We also relate 
local norms to the boundary values of the eigenvalue problem. 
In Section \ref{sec:gmsfem} we review a very particular realization of the GMsFEM 
methodology that is the one analyzed in this paper. In Section \ref{section8} 
we obtain our interpolation  error for the resulting  method. We also write our 
convergence result. We present some numerical experiments in Section \ref{sec:numerics}. Our numerical results verify our theoretical findings for smooth coefficients. We also consider a more practical case with heterogeneous multiscale coefficients.  Finally, in Section \ref{sec:discussions} we draw some conclusions
and make some final comments.

\section{Preliminaries on multiscale finite element methods}\label{sec:preliminaries}

In this section, we  describe multiscale finite element method
framework. 
In general terms, the MsFEMs compute
the coarse-scale solution by using multiscale basis functions. It can be 
casted as a numerical upscaling procedure. Also as a numerical homogenization 
method where, instead of effective parameters representing small scale effects, 
basis functions are constructed that capture the small scale effects on solutions.

Multiscale techniques can be applied to variety of 
problems. In this paper, in order to fix ideas, we consider 
a second order elliptic problem with a possible multiscale high-contrast
coefficient.
More precisely, let $\Omega\subset\mathbb{R}^2$ (or $\mathbb{R}^3$) be a 
polygonal domain.
We consider the elliptic equation with heterogeneous coefficients
\begin{equation*}
-\mbox{div}(\kappa(x)\nabla u)=f,
\end{equation*}
where $\kappa(x)$ is a heterogeneous scalar field with high-contrast. In particular,
we assume that $\kappa(x)\geq c_0>0$ (bounded below), while
$\kappa(x)$ can have very large values. We assume that $\kappa\ in L^\infty({\Omega})$ and therefore 
$\kappa$ might be discontinuous. 
The
variational formulation of this problem is: 
Find $u\in H_0({\Omega})$ such that 
\begin{equation}\label{eq:problem}
a(u,v)=f(v) \quad \mbox{ for all } v\in H_0^1({\Omega}).
\end{equation}
Here the bilinear form $a$ and the linear functional $f$ are 
defined by
\begin{equation*}
a(u,v)=\int_{\Omega} 
\kappa(x)\nabla u(x)\nabla v(x) dx 
\quad \mbox{ for all } u,v\in H_0^1({\Omega})
\end{equation*}
and 
\[
f(v)=\int_{\Omega}f(x)v(x)dx \quad \mbox{ for all } v\in H_0^1({\Omega}).
\]

 Let $\mathcal{T}^H$ be a triangulation composed by elements
 $K$. We refer to the triangulation $\mathcal{T}^H$
as a coarse triangulation in the sense that does not necessarily  
resolve all the scales 
in the model (in our case that would be all variations and discontinuities 
of $\kappa$). We denote  $\{y_i\}_{i=1}^{N_v}$ the vertices of the coarse mesh
$\mathcal{T}^H$ and define the neighborhood of the 
node $y_i$ by
\begin{equation*}
\omega_i=\bigcup \{ K\in\mathcal{T}^H; ~~~ y_i\in \overline{K}\}.
\end{equation*}
and the neighborhood of an element $K$ by, 
\begin{equation}\label{eq:def:omegaK}
\omega^K=\bigcup\{ w_i~~; ~~~ K \subset w_i\}.
\end{equation}

Using the coarse mesh $\mathcal{T}^H$ 
we introduce  coarse
 basis functions 
$\{\Phi_i\}_{i=1}^{N_c}$, where $N_c$ is the number 
of coarse basis functions. In our paper, the basis functions
are supported in $\omega_i$; however, for $\omega_i$, there
may be multiple basis functions.
MsFEMs approximate the solution on a coarse grid as
$u_0=\sum c_i \Phi_i$, where $c_I$ are determined from 
\begin{equation*}
a(u_0,v)=f(v),\quad \text{for all}\  v\in 
\text{span}\{ \Phi_i\}_{i=1}^{N_c}.
\end{equation*}
Once $c_i$'s are determined, one can define a fine-scale approximation
of the solution by reconstructing via basis functions,
$u_0=\sum_{i=1}^{N_c} c_i \Phi_i$.

\section{Global eigenvalue problem}\label{sec:global}

In this section, we recall some facts about the global eigenvalue problem associated 
to problem \eqref{eq:problem1}. We stress that the global eigenvalue problem 
is used in the analysis only and it is not use in the computations.

We start the presentation by introducing the global mass bilinear form. This is given by
\begin{equation*}
m(v,w)=\int_{\Omega}\kappa v w \quad \mbox{ for all } v,w\in H^1_0(\Omega).
\end{equation*}
Note that we use the coefficient $\kappa$ in the mass matrix. The reason is 
that our main application in mind is on high-contrast problems and, as show in 
\cite{Review, EG09, ge09_1, ge09_1reduceddim}, it is important to define the mass 
matrix with the coefficient $\kappa$. Moreover, more complicated bilinear forms 
can be also used as in recent developments in GMsFEM; see \cite{egh12, EfendievGLWESAIM12}. 

We consider the eigenvalue problem (in weak form) that seeks to find eigenfunctions 
$\phi$ and scalars $\mu$ such that
\begin{equation}\label{eq:eigenvalueproblem3}
a(\phi,z)=\mu m(\phi,z)\quad \mbox{ for all } z\in H_0^1({\Omega}).
\end{equation}
Denote it's eigenvalues and eigenfunctions by
$\{\mu_\ell\}$ and $\{ \phi_{\ell}\}$, 
respectively.
We order eigenvalues as 
\begin{equation}\label{eq:orderingg}
\mu_1\leq \mu_2 \leq \dots\leq 
\mu_\ell \dots.
\end{equation}
We have $\mu_1>0$. The eigenvalue  problem (\ref{eq:eigenvalueproblem3})  is the weak form of the eigenvalue problem 
\begin{equation}\label{eq:eigproblem}
-\mbox{div} (\kappa\nabla \phi )=\mu \kappa \phi
\end{equation}
in ${\Omega}$ with  homogeneous 
Dirichlet boundary condition on $\partial {\Omega}$.

We recall that the eigenvectors form a complete ($m$)-orthonormal system of $L^2({\Omega})$
that is also orthogonal with respect to the bilinear form $a$.  Given any $v\in H^1_0(\Omega)$ we can write 
\[
v=\sum_{\ell =1}^{\infty} 
m(v,\phi_\ell) \phi_\ell
\]
and compute the bilinear form $a$ as
\begin{equation}\label{eq:avvg}
a (v,v)=\sum_{\ell =1}^{\infty}
 \mu_{\ell }
 m(v,\phi_\ell)\phi_\ell
\end{equation}
and the bilinear form $m$ as
\begin{equation}\label{eq:mvvg}
m(v,v)=
\sum_{\ell=1}^{\infty}
 m(v,\phi_\ell)^2 .
\end{equation}

It is important to recall that the expansion of the solution can be explicitly given. In fact, 
from the weak form (\ref{eq:problem}) we see that 
\[
\mu_\ell m(u,\phi)=a(u,\phi)=f(\phi).
\]
Then we have
\begin{equation}\label{eq:expansionug}
u=\sum_{\ell=1}^\infty \frac{1}{\mu_\ell}f(\phi_\ell)\phi_\ell.
\end{equation}

The eigenvector are the regular functions \emph{par excellence} 
when working with the differential operator $-\mbox{div}(\kappa \nabla \cdot)$. In particular 
we stress the following fact. 
\begin{remark}\label{rem:globaleig}
We have that $\kappa \nabla \phi_\ell$ has a square integrable divergence. That is, 
$-\mbox{\normalfont div}(\kappa\nabla \phi_\ell)\in L^2({\Omega})$.
This follows by observing that, in a generalize sense, $-\mbox{\normalfont div}(\nabla \phi_\ell)=\mu_\ell \kappa \phi_\ell$.
\end{remark}

\subsection{Global Norms based on eigenvalue expansion decay}\label{sec:globalnorms}

In this section, we introduce a scale of norms that help measuring the 
decay of the expansions in terms of eigenvectors of the global eigenvalue problem. 
These norms are  used in the {\it a priori} error estimates of our 
GMsFEM method. We note that, without assuming some sort of regularity of the solution
of \eqref{eq:problem},
it is difficult to measure the rate of the error in finite element approximations and give 
error estimates.  For this paper, we only assume that the forcing term is
square integrable in order to obtain approximation using global eigenvector. Later we consider the case of approximation using locally constructed basis functions with small support that employ local eigenvector in its design. \\

For any $v\in L^2({\Omega})$ written as $v=\sum_{\ell=1}^\infty m(v,\phi_\ell) \phi_\ell$
and $s>0$,  we introduce the norm $ ||| \cdot ||| _{s;{\Omega}}$ defined by,
\begin{equation*}
|||v|||_{s;{\Omega}}^2=\sum_{\ell=1}^\infty  \mu_\ell^{s} m(v,\phi_\ell)^2.
\end{equation*}
We note that these norms depend on the bilinear forms $a$ and $m$ but, in order to 
make notation simpler, we do not stress this dependence in our notation. Note that 
\[
|||v|||_{0;{\Omega}}^2=m(v,v)=\int_{\Omega} \kappa v^2 \quad \mbox{ and }\quad 
|||v|||_{1;{\Omega}}^2=a(v,v)=\int_{\Omega} \kappa |\nabla v|^2.
\]
In this paper, we mainly use the norm  $ ||| \cdot ||| _{s; {\Omega}}$  with $s=2$. We have that 
\[
|||u|||_{2;{\Omega}}^2=\sum_{\ell=1}^\infty \mu_\ell ^2m(v,\phi_\ell)^2.
\]
Then, if $|||u|||_{2;{\Omega}}<\infty$ we can define the operator $\mathcal{A}$ applied to $u$ by
\[
\mathcal{A}u=\sum_{\ell=1}^\infty \mu_\ell m(u,\phi_\ell) \phi_\ell.
\]
We readily have  $\mathcal{A}u\in L^2({\Omega})$ since,
\[
||| \mathcal{A}u|||_{0;{\Omega}}^2=m(\mathcal{A}u,\mathcal{A}u)=\sum_{\ell=1}^\infty \mu_\ell^2 m(u,\phi_\ell)^2=|||u|||_{2;{\Omega}}^2<\infty.
\]
Furthermore, if  $|||u|||_{2;{\Omega}}<\infty$ we also have the following \emph{integration by parts} relation, that can be verified by straightforward calculations,
\begin{equation}\label{eq:regularsolution}
a(u,v)=m(\mathcal{A}u,v) \quad \mbox{ for all } v\in H_0^1({\Omega}).
\end{equation}

We now present a characterization of $\mathcal{A}$ using the divergence operator 
$\mbox{div}$ and the coefficient  $\kappa$. This implies that for even integer values of $s$  (in particular for $s=2$), the norm  $||| \cdot |||_s$ is computed by subassembly of 
similar norms in subdomains.   

\begin{theorem}\label{thm:AanddivGlobal}

The operator $\mathcal{A}$ is a locally defined operator. More precisely, if $|||u|||_{2;{\Omega}}<\infty$ we have 
that $-\kappa^{-1}\mbox{\normalfont div}(\kappa \nabla u)$ belongs to the space $L^2({\Omega})$ and we have 
\[
\mathcal{A} u = -\kappa^{-1}\mbox{\normalfont div}(\kappa \nabla u).
\]
Moreover, we have 
\[
|||u|||^2_{2;{\Omega}}=\| \mathcal{A}u\|_0^2=\int_{\Omega} \kappa^{-1} | \mbox{\normalfont div}(\kappa \nabla u)|^2.
\]
\end{theorem}
\begin{proof}
Recall that for $u\in L^2({\Omega})$,  we have the expansion
$u=\sum_{\ell =1}^{\infty} 
m(u,\phi_\ell) \phi_\ell.$ 
For an integer  $N$ define the truncated approximation of $u$ as,
\[
u^N=\sum_{\ell =1}^{N} 
m(u,\phi_\ell) \phi_\ell.
\]
We construct the $-\kappa^{-1}\mbox{div}(\kappa \nabla u)$ as a limit in the $m-$norm of the sequence 
of rescaled divergences given by  $\{ \kappa^{-1} \mbox{div}(\kappa \nabla u^N) \}_{N=1}^\infty$.
To this end, we prove that the sequence $\{ \kappa^{-1} \mbox{div}(\kappa \nabla u^N) \}_{N=1}^\infty$ is a  Cauchy sequence in the $m-$norm. 
Indeed, we have, by using the eigenvalue problem (\ref{eq:eigproblem}) and Remark \ref{rem:globaleig}, the following identity,
\[
\kappa^{-1}\mbox{div}(\kappa \nabla u^N) =\sum_{\ell =1}^{N} 
m(u,\phi_\ell)\mu_\ell \phi_\ell.
\]
So that, using the orthogonality of the eigenvectors, we conclude that for every $M>N$ we have, 
\[
\| \kappa^{-1}\mbox{div}(\kappa \nabla u^M)-\kappa^{-1}\mbox{div}(\kappa \nabla u^N)\|_0^2=
\sum_{\ell=N+1}^{M} m(u,\phi_\ell)^2 \mu_\ell^2.
\]
This implies the claim since the series $|||u|||^2_{2;{\Omega}}=\sum_{\ell=1}^{\infty} m(u,\phi_\ell)^2 \mu_\ell^2<\infty$. We conclude that there exist an $L^2({\Omega})$ function, denoted by $U$,  such that we have $\| U+\kappa^{-1}\mbox{div}(\kappa \nabla u^N)\|\to 0$ 
when $N\to \infty$. We also have that for any  $z\in H^1_0({\Omega})$ it holds, 
\[
\int_{\Omega}\kappa Uz = -\lim_{N\to \infty} \int_{\Omega} \kappa \kappa^{-1}\mbox{div}(\kappa \nabla u^N) z= \int_{\Omega} \kappa \nabla u \nabla z,
\]
which proves that $U= -\kappa^{-1}\mbox{div}(\kappa \nabla u^N) z$.

Finally note that, by using \eqref{eq:regularsolution},  for every function $z\in H^1_0({\Omega})$ we have, 
\[
\int_{\Omega} \kappa \mathcal{A} u z =m(\mathcal{A} u, v)=a(u,z)=\int_{\Omega} \kappa \nabla u \nabla z.
\]
\end{proof}

\begin{remark}
Notice that if $u$ is the solution of \eqref{eq:problem} with $f\in L^2({\Omega})$,  then,  we have  
$\mathcal{A}u=\kappa^{-1}f$. 
\end{remark}
\begin{lemma}\label{lem:aprioriest}
Assume that $|||f|||_s^2<\infty$ and let $u$ be the solution of~\eqref{eq:problem}. 
Consider $t\geq 1$ such that  $t-s-2\leq 0$ we have
\begin{equation*}
|||u|||^2_t\leq \mu^{t-s-2}_1 |||f|||_s^2
\end{equation*}
In particular, if $|||f|||_0<\infty$ we have that $|||u|||_2=|||\kappa^{-1}f|||_0=
\int_{\Omega} \kappa^{-1}f^2$.
\end{lemma}
\begin{proof}
Using the explicit expansion in~\eqref{eq:expansionug}, the definition of the 
norm $\|\cdot\|_t$ and then increasing the order of eigenvalues we have
\begin{align*}
||| u |||^2_t 
&=\sum_{\ell=1}^\infty \mu^{t}_\ell \frac{1}{\mu_\ell^2}f(\phi_\ell)^2
=\sum_{\ell=1}^\infty \mu^{t-s-1}_\ell\mu^{s}_\ell f(\phi_\ell)^2\\
&\leq
\mu^{t-s-2}_0 \sum_{\ell=1}^\infty  \mu^{s}_\ell f(\phi_\ell)^2\\
&=
\mu^{t-s-2}_0 \sum_{\ell=1}^\infty  \mu^{s}_\ell 
\left(\int_{\Omega} f \phi_\ell\right)^2\\
&=
\mu^{t-s-2}_0 \sum_{\ell=1}^\infty  \mu^{s}_\ell\left( \int_{\Omega}\kappa  (\kappa^{-1}f )\phi_\ell\right)^2
=
\mu^{t-s-2}_0 |||\kappa^{-1}f|||_s^2.
\end{align*}
This finishes the proof.
\end{proof}

\subsection{Approximation using global eigenvectors}

In this section, we show how to obtain a priori error estimates if we use the space spanned by 
the first eigenvectors.
Given an integer $L$ and $v\in H^1_0(\Omega)$, we define
\begin{equation*}
\mathcal{J}_{L} v=\sum_{\ell=1}^{L}
 m(v,\phi_\ell)\phi_\ell. 
\end{equation*}
From (\ref{eq:orderingg}),
(\ref{eq:avvg}),  and (\ref{eq:mvvg}) it is easy to prove the following inequality 

\begin{equation}\label{eq:truncation1}
\int_{\Omega} \kappa(v-\mathcal{J}_{L} v)^2\leq 
\frac{1}{\mu_{L+1}} a(v-\mathcal{J}_{L} v,v-
\mathcal{J}_{L} v)
\leq 
\frac{1}{\mu_{L+1}} a(v,v).
\end{equation}

When ${L}=1$ and $\kappa=1$ we obtain the usual 
Friedrichs' inequality.\\

 If $u$ is the solution of~\eqref{eq:problem} and $||| f|||_s<\infty$ we have
the following a priori estimate.
\begin{lemma}
Let $u$ be the solution of~\eqref{eq:problem}. If if $t-s-2<0$ we have
\begin{equation*}
|||u-\mathcal{J}_{L} u|||^2_t\leq 
\mu^{t-s-2}_{L+1} |||f|||_s^2.
\end{equation*}
\end{lemma}
\begin{proof}
Using the explicit expansion in~\eqref{eq:expansionug}, the definition of the 
norm $\|\cdot\|_t$ and then increasing oder of eigenvalues we have
\begin{align*}
|||u-\mathcal{J}_{L} u|||^2_t
&=\sum_{\ell={L}+1}^\infty \mu^{t}_\ell \frac{1}{\mu_\ell^2}f(\phi_\ell)^2
=\sum_{\ell=L+1}^\infty \mu^{t-s-2}_\ell\mu^{s}_\ell f(\phi_\ell)^2\\
&\leq
\mu^{t-s-2}_{L+1} \sum_{\ell=1}^\infty  \mu^{s}_\ell f(\phi_\ell)^2=
\mu^{t-s-2}_{L+1} |||f|||_s^2.
\end{align*}
This finishes the proof.
\end{proof}

We observe that, in particular, we have the following a priori error estimates. 
\begin{align*}
a(u-\mathcal{J}_{L} u,u-\mathcal{J}_{L} u)
=|||u-\mathcal{J}_{L} u|||^2_1\leq 
\lambda^{-(s+1)}_{L+1} |||f|||_s^2.
\end{align*}

The space generated by the first eigenvalues gives good approximation spaces and 
the analysis becomes easy. 

\section{Dirichlet eigenvalue problem in coarse blocks}\label{sec:localD}
In this section, we study the local Dirichlet eigenvalue problem associated 
to problem \eqref{eq:problem1}.
For any $K$, we  define the following bilinear forms
\begin{equation*}
a^{K}(v,w)=\int_{K} \kappa \nabla v \nabla w \quad \mbox{ for all } v,w\in H^1(K), \quad i=1,\dots,N,
\end{equation*}
and 
\begin{equation*}
m^{K}(v,w)=\int_{K}\kappa v w \quad \mbox{ for all } v,w\in H^1(K).
\end{equation*}
We consider the eigenvalue problems that seek 
eigenfunctions $\phi\in H^1_0(K)$ and scalars $\mu$ such that
\begin{equation*}
a^{K}(\phi,z)=\mu m^{K}(\phi,z)\quad 
\mbox{ for all } z\in H^1_0(K).
\end{equation*}
and denote its eigenvalues and eigenvectors by
$\{\mu_\ell^{K}\}$ and $\{ \phi_{\ell}^{K}\}$, 
respectively. Note that the eigenvectors 
$\{\phi^{K}_\ell\}$ form an orthonormal basis of 
 $L^2(K)$ with respect to the $m^K$ inner product.
We order eigenvalues as 
\begin{equation*}
\mu_1^{K}<\mu_2^{K} \leq \dots\leq 
\mu_\ell^{K} \dots.
\end{equation*}
The eigenvalue  problem above  
corresponds to the approximation 
of the eigenvalue problem 
\begin{equation}\label{eq:localeigproblemD}
-\mbox{div} (\kappa\nabla \phi  )=\mu \kappa \phi \quad \mbox{ in } K,
\end{equation}
with  homogeneous Dirichlet boundary condition on $\partial K$. These eigenvectors 
are the model of regular functions working with the differential operator
$-\mbox{div}(\kappa \nabla \cdot)$. In particular, the operator is well defined 
and well behaved  over these functions. We have that $\kappa \nabla \phi^K_\ell$ has 
a square integrable divergence. That is, 
$-\mbox{\normalfont div}(\nabla \phi^K_\ell)\in L^2(K)$. This follows by observing that, 
in a generalize sense, $-\mbox{\normalfont div}(\nabla \phi_\ell)=\mu_\ell 
\kappa \phi_\ell$.\\

Now we use the expansion in terms of local eigenvectors  and define norms based on the decay of the eigenexpansion. 
These local norms are the ones that naturally appear in the local interpolation errors for our interpolation operator. 
A main issue is to compare this local norms with the global norms defined in Section \ref{sec:globalnorms} for $s>1$. 
For the case $s=2$, we prove in Theorem \ref{thm:comparison} that the local norms can be  assembled to obtain 
a global norms equivalent to the norm $|||\cdot |||_{2,\Omega}$ only for functions that have zero value on block boundaries.

Given any $v\in L^2(K)$ we can write 
\[
v=\sum_{\ell =1}^{\infty} 
m^{K}(v,\phi_\ell^{K}) \phi_\ell^{K} 
\]
and compute the local energy bilinear form by
\begin{equation*}
a^{K} (v,v)=\sum_{\ell =1}^{\infty}
 m^{K}(v,\phi_\ell^{K})^2 \mu_{\ell }^{K}.
\end{equation*}
We can also compute the local mass bilinear as,
\begin{equation*}
m^{K}(v,v)=
\sum_{\ell=1}^{\infty}
 m^{K}(v,\phi_\ell^{K})^2 .
\end{equation*}

The local norm to measure the decay of the expansion is introduced as follows. 
We 
introduce the norm,
\begin{equation*}
||| v|||_{s,K}^2= \sum_{\ell=1}^\infty 
(\mu_\ell^{K})^{s} m^{K}( v, \phi_\ell^{K})^2.
\end{equation*}
Note that \[
|||v|||_{0;K}^2=m(v,v)=\int_{K} \kappa v^2 \quad \mbox{ and }\quad 
|||v|||_{1;K}^2= \int_{K} \kappa |\nabla v|^2.
\]
We consider the case $s=2$. If $|||u|||_{2,K}<\infty$ we can define the operator 
$\mathcal{A}^{K}$ by,
\[
\mathcal{A}^{K}u= 
\sum_{\ell=1}^\infty  
\mu_\ell^{K} m^{K}(u,\phi_\ell^{K})\phi_\ell^{K}
\]
which is square integrable since
\[
||| \mathcal{A}^{K}u|||_{0,K}^2=
\sum_{\ell=1}^\infty  
(\mu_\ell^{K})^2 m^{K}(u,\phi_\ell^{K})^2=
|||u|||_{2,K}^2.
\]
Additionally  if $|||u|||_{2,K}<\infty$,  we have the following \emph{local integration by parts} relation that can be verified by direct calculations,
\[
a^{K}(u,v)=m^{K}(\mathcal{A}^{K}u,v) \mbox{ for all } v\in H^1_0(K).
\]

We have the following result. This result reveals that the local norms $||| u|||_{2,K}$ is related to the $L^2$ integrability of $\mbox{div}(\kappa \nabla \cdot u)$. The proof of the following theorem 
follows the proof of Theorem \ref{thm:AanddivGlobal} but we presented in the local setting in the interest of completeness.

\begin{theorem}\label{thm:comparison}

The operator $\mathcal{A}^{K}$ is a locally defined operator. More precisely, if $|||u|||^2_{2,K}<\infty$ we have 
that $-\kappa^{-1}\mbox{\normalfont div}(\kappa \nabla u)$ belongs to the space $L^2(K)$ and we have 
\[
a^K( u,v) = -\int_K\kappa^{-1}\mbox{\normalfont div}(\kappa \nabla u)v \quad 
\mbox{ for all } v\in H_0^1(K).
\]
Moreover, we have
\[
|||u|||^2_{2,K}=\| \mathcal{A}^{K}u\|_{0,K}^2= \int_{K} \kappa^{-1} | \mbox{\normalfont div}(\kappa \nabla u)|^2.
\]
\end{theorem}
\begin{proof}
Since $u\in L^2(K)$ we have the expansion  $u=\sum_{\ell =1}^{\infty} m^{K}(u,\phi_\ell^{K}) \phi_\ell^{K}$.
For any integer $N$, truncate this expansion to get,
\[
u^N:=\sum_{\ell =1}^{N} 
m^{\omega_i}(u,\phi_\ell^{K}) \phi_\ell^{K}.
\]
The sequence of rescaled divergences, $\{ \kappa^{-1} \mbox{div}(\kappa \nabla u^N) \}_{N=1}^\infty$, is a 
Cauchy sequence in the $m^{K}-$norm. 
Indeed, we have, by using the eigenvalue problem (\ref{eq:localeigproblemD}), the following identity,
\[
\kappa^{-1}\mbox{div}(\kappa \nabla u^N) =\sum_{\ell =1}^{N} 
m^{K}(u,\phi_\ell^{K})\mu_\ell^{K} \phi_\ell^{K}.
\]
So that, using the orthogonality of the eigenvectors we conclude that for every $M>N$ we have, 
\[
||| \kappa^{-1}\mbox{div}(\kappa \nabla u^M)-\kappa^{-1}\mbox{div}(\kappa \nabla u^N)|||_{0,{K}}^2=
\sum_{\ell=N+1}^{M} m^{K}(u,\phi_\ell^{K})^2(\mu_\ell^{K})^2.
\]
Then, there exists an $L^2(K)$ function, say $U$, such that
$\| U+\kappa^{-1}\mbox{div}(\kappa \nabla \mu)\|_{0,{K}}\to 0$ 
when $N\to \infty$.

We also have that for any  $z\in H^1_0(K)$ it holds, 
\[
\int_{K}\kappa Uz = -\lim_{N\to \infty} \int_{K} \kappa \kappa^{-1}\mbox{div}(\kappa \nabla \mu) z= \int_K \kappa \nabla u \nabla z,
\]
which proves that $U= -\kappa^{-1}\mbox{div}(\kappa \nabla \mu) $.

Finally note that for every function $z\in H^1_0(K)$ we have, 
\[
\int_K \kappa \mathcal{A}^{K} u z =m(\mathcal{A}^{K} u, v)=a(u,z)=\int_K \kappa \nabla u \nabla z.
\]
\end{proof}

Given an integer $L$ and $v\in H^1_0(K)$, we define
\begin{equation*}
\mathcal{J}_{L}^{K} v=\sum_{\ell=1}^{L}
 m(v,\phi_\ell)\phi_\ell. 
\end{equation*}
From the analogous to (\ref{eq:orderingg}),
(\ref{eq:avvg}),  and (\ref{eq:mvvg}) it is easy to prove the following inequality

\begin{lemma}\label{lem:aprioriestD}
Assume that $u\in H^1_0(K)$ and  $||| u|||_{s,K}\leq \infty$  with $s>1$. We have for $1\leq s\leq t\leq 2$,
\begin{equation*}
|||u-\mathcal{J}_{L}^{K} u|||^2_{t,K}\leq 
(\mu^K_{L+1})^{t-s} |||u|||_{s,K}^2.
\end{equation*}
In particular, 
\begin{equation*}
|||u-\mathcal{J}_{L}^{K} u|||^2_{1,K}\leq 
\mu^K_{L+1} |||\mathcal{A}^Ku|||_{0,K}^2.
\end{equation*}
\end{lemma}

\section{Local Neumann eigenvalue problem in coarse neighborhoods}\label{sec:localN}
In this section, we study local eigenvalue problem associated 
to problem \eqref{eq:problem1}.
For any $\omega_i$, we  define the following bilinear forms

\begin{equation*}
a^{\omega_i}(v,w)=\int_{\omega_i} \kappa \nabla v \nabla w \quad \mbox{ for all } v,w\in H^1(\omega_i), \quad i=1,\dots,N,
\end{equation*}

and 

\begin{equation*}
m^{\omega_i}(v,w)=\int_{\omega_i}\kappa v w \quad \mbox{ for all } v,w\in H^1(\omega_i).
\end{equation*}

Define $\widetilde{V}(\omega_i)=\{ v\in H^1(\omega_i) ~:~ v=0 \mbox{ on } 
\partial\omega_i\cap\partial\Omega\}$ if $\partial\omega_i\cap\partial \Omega$  is non-empty and  
$\widetilde{V}(\omega_i)=\{ v\in H^1(\omega_i) ~:~ \int_{\omega_i} v=0\}$ otherwise. 
We consider the eigenvalue problems that seek 
eigenfunctions $\psi\in  \widetilde{V}(\omega_i)$ and scalars $\lambda$ such that
\begin{equation*}
a^{\omega_i}(\psi,z)=\lambda m^{\omega_i}(\psi,z)\quad 
\mbox{ for all } z\in \widetilde{V}(\omega_i),
\end{equation*}
and denote its eigenvalues and eigenvectors by
$\{\lambda_\ell^{\omega_i}\}$ and $\{ \psi_{\ell}^{\omega_i}\}$, 
respectively. Note that the eigenvectors 
$\{\psi^{\omega_i}_\ell\}$ form an orthonormal basis of 
of $L^2(\omega_i)$ with respect to the $m^{\omega_i}$ inner product. 
Note that $\lambda_1^{\omega_i}=0$ when $\partial \omega_i \cap \partial {\Omega}$ is empty, that is, when 
$\omega_i$ is a floating subdomain. 
We order eigenvalues as 
\begin{equation*}
\lambda_1^{\omega_i}<\lambda_2^{\omega_i} \leq \dots\leq 
\lambda_\ell^{\omega_i} \dots.
\end{equation*}
The eigenvalue  problem above  
corresponds to the approximation 
of the eigenvalue problem 
\begin{equation}\label{eq:localeigproblemN}
-\mbox{div} (\kappa\nabla v  )=\lambda \kappa v \quad \mbox{ in }\omega_i
\end{equation}
 with  homogeneous 
Neumann boundary condition on $\partial\omega_i\cap \Omega$ and homogeneous 
Dirichlet boundary condition on $\partial \omega_i\cap \partial \Omega$ (when non-empty).\\

As mentioned before when studying the global eigenvalue problem, these eigenvectors are the model of regular functions working with the differential operator $-\mbox{div}(\kappa \nabla \cdot)$. In particular the operator is well defined and well behaved 
over these functions.
We have that $\kappa \nabla \psi_\ell$ has a square integrable divergence. That is, 
$-\mbox{\normalfont div}(\nabla \psi_\ell)\in L^2(\omega_i)$.
This follows by observing that, in a generalize sense, $-\mbox{\normalfont div}(\nabla \psi_\ell)=\lambda_\ell \kappa \psi_\ell$.\\

Given any $v\in  \widetilde{V}(\omega_i)$ we can write 
\[
v=\sum_{\ell =1}^{\infty} 
m^{\omega_i}(v,\psi_\ell^{\omega_i}) \psi_\ell^{\omega_i} 
\]
and compute the local energy bilinear form by
\begin{equation*}
a^{\omega_i} (v,v)=\sum_{\ell =1}^{\infty}
 m^{\omega_i}(v,\psi_\ell^{\omega_i})^2 \lambda_{\ell }^{\omega_i}.
\end{equation*}
We can also compute the local mass bilinear as,
\begin{equation*}
m^{\omega_i}(v,v)=
\sum_{\ell=1}^{\infty}
 m^{\omega_i}(v,\psi_\ell^{\omega_i})^2 .
\end{equation*}

The local norm to measure the decay of the expansion is introduced as follows. 
We 
introduce the semi-norm,
\begin{equation*}
||| v|||_{s,\omega_i}^2=\sum_{\ell=1}^\infty 
(\lambda_\ell^{\omega_i})^{2s} m^{\omega_i}( v, \psi_\ell^{\omega_i})^2.
\end{equation*}
Note that for $s=0$ we have a norm and for $s=1$ the semi-norms becomes a norm when restricted to non-constant functions on $\omega_i$, more precisely, \[
|||v|||_{0;\omega_i}^2=m(v,v)=\int_{\omega_i} \kappa v^2 \quad \mbox{ and }\quad 
|||v|||_{1;\omega_i}^2=  \int_{\omega_i} \kappa |\nabla v|^2.
\]
We consider the case $s=2$. If $|||u|||_{2,\omega_i}<\infty$ we can define the operator 
$\mathcal{A}^{\omega_i}$ by,
\begin{equation*}
\mathcal{A}^{\omega_i}u= 
\sum_{\ell=1}^\infty  
\lambda_\ell^{\omega_i} m^{\omega_i}(v,\psi_\ell^{\omega_i})\psi_\ell^{\omega_i}.
\end{equation*}
which is square integrable since
\[
||| \mathcal{A}^{\omega_i}u|||_{0,\omega_i}^2=
\sum_{\ell=1}^\infty  
(\lambda_\ell^{\omega_i})^2 m^{\omega_i}(v,\psi_\ell^{\omega_i})^2=
|||u|||_{2,\omega_i}^2.
\]
Additionally,  if $|||u|||_{2,\omega_i}<\infty$,  we have the following \emph{local integration by parts} relation (that can be verified directly by the series expansion of both sides),
\begin{equation}\label{eq:localbypartsintegration}
a^{\omega_i}(u,v)=m^{\omega_i}(\mathcal{A}^{\omega_i}u,v) \mbox{ for all } v\in H^1(\omega_i).
\end{equation}

We have the following result.
\begin{theorem}\label{thm:AanddivLocal2}
The operator $\mathcal{A}^{\omega_i}$ is a locally defined operator. More precisely, if $|||u|||^2_{2,\omega_i}<\infty$ we have 
that $-\kappa^{-1}\mbox{\normalfont div}(\kappa \nabla u)$ belongs to the space $L^2(\omega_i)$ and we have 
\[
a^{\omega_i}( u,v) = -\int_{\omega_i}\kappa^{-1}\mbox{\normalfont div}(\kappa \nabla u)v \quad 
\mbox{ for all } v\in H_0^1(\omega_i).
\]
Moreover, we have
\[
|||u|||^2_{2,\omega_i}=\| \mathcal{A}^{\omega_i}u\|_{0,\omega_i}^2= \int_{\omega_i} \kappa^{-1} | \mbox{\normalfont div}(\kappa \nabla u)|^2.
\]
\end{theorem}
\begin{proof}
Since $u\in L^2(\omega_i)$ we have the expansion  $u=\sum_{\ell =1}^{\infty} m^{\omega_i}(u,\psi_\ell^{\omega_i}) \psi_\ell^{\omega_i}$.
For any integer $N$, truncate this expansion to get,
\[
u^N:=\sum_{\ell =1}^{N} 
m^{\omega_i}(u,\psi_\ell^{\omega_i}) \psi_\ell^{\omega_i}.
\]
The sequence of rescaled divergences, $\{ \kappa^{-1} \mbox{div}(\kappa \nabla u^N ) \}_{N=1}^\infty$, is a 
Cauchy sequence in the $m^{\omega_i}-$norm. 
Indeed, we have, by using the eigenvalue problem (\ref{eq:localeigproblemN}), the following identity,
\[
\kappa^{-1}\mbox{div}(\kappa \nabla u_N) =\sum_{\ell =1}^{N} 
m^{\omega_i}(u,\psi_\ell^{\omega_i})\lambda_\ell^{\omega_i} \psi_\ell^{\omega_i}.
\]
So that, using the orthogonality of the eigenvectors we conclude that for every $M>N$ we have, 
\[
||| \kappa^{-1}\mbox{div}(\kappa \nabla u_M)-\kappa^{-1}\mbox{div}(\kappa \nabla u_N)|||_{0,{\omega_i}}^2=
\sum_{\ell=N+1}^{M} m^{\omega_i}(u,\psi_\ell^{\omega_i})^2(\lambda_\ell^{\omega_i})^2.
\]
Then, there exists an $L^2(\omega_i)$ function, say $U$, such that
$\| U+\kappa^{-1}\mbox{div}(\kappa \nabla u_N)\|_{0,{\omega_i}}\to 0$ 
when $N\to \infty$. We also have that for any  $z\in H^1_0(\omega_i)$ it holds, 
\[
\int_{\omega_i}\kappa Uz = -\lim_{N\to \infty} \int_{\omega_i} \kappa \kappa^{-1}\mbox{div}(\kappa \nabla u_N) z= \int_{\omega_i} \kappa \nabla u \nabla z,
\]
which proves that $U= -\kappa^{-1}\mbox{div}(\kappa \nabla u_N) $. Finally, note that 
for every function $z\in H^1_0(\omega_i)$ we have, 
\[
\int_{\omega_i} \kappa \mathcal{A}^{\omega_i} u z =m(\mathcal{A}^{\omega_i} u, v)=a(u,z)=\int_{\omega_i} \kappa \nabla u \nabla z.
\]
\end{proof}
\begin{remark}\label{rem:nullnormalflux}
In virtue of Theorem \ref{thm:AanddivLocal2} and 
the equality \eqref{eq:localbypartsintegration} we see that a necessary condition for 
the $L^2$ integrability 
of $\mbox{div}(\kappa\nabla u)$ is that $\kappa\partial_\eta u=0$ on 
$\partial \omega_i$. Note that in general, if we are not sure
$|||u|||^2_{2,\omega_i}<\infty$
and we do not assume  $\partial_\eta u=0$ then, the integration by parts become 
\[
a^{\omega_i}(u,v)-\int_{\partial\omega_i} \kappa\partial_\eta u v= -
\int_D \mbox{div}(\kappa\nabla u)v \mbox{ for all } v\in H^1(\omega_i).
\]
Therefore, doing estimates about the eigenvalue decay is harder in this case. We also mention that in the analysis presented in \cite{egw10}, it is assume the square integrability of 
$\mbox{\normalfont div}(\kappa\nabla u)$ that, as mentioned above,  implies $\kappa\partial_\eta u=0$. Assuming that the solution has zero flux across boundaries neighborhoods 
is not a general assumption. A main contribution of this paper is to clarify this main assumption of 
\cite{egw10} and to present an analysis valid for the general case 
where the solution $u$ does not have null fluxes across neighborhood boundaries. As we show in our numerical experiments, for the case of a solution that is not close to a function with null flux across neighborhood boundaries, the convergence rate of the GMsFEM as introduced in \cite{egw10} is not optimal and additional basis functions constructed from local Dirichlet eigenvalues must be introduced to recover good convergence.
\end{remark}

\begin{lemma}\label{lem:aprioriestN}
Assume that $u\in H^1(\omega_i)$, $\partial_n u = 0$ and $||| u|||_{s,\omega_i}\leq \infty$ with $s>1$.  We have for $1\leq s\leq t\leq 2$,
\begin{equation*}
|||u-\mathcal{I}_{L}^{\omega_i} u|||^2_{t,\omega_i}\leq 
(\lambda_{L+1}^{\omega_i})^{t-s} |||u|||_{s,\omega_i}^2.
\end{equation*}
In particular, 
\begin{equation*}
|||u-\mathcal{I}_{L}^{\omega_i} u|||^2_{1,\omega_i}\leq 
\lambda_{L+1}^{\omega_i} |||\mathcal{A}^{\omega_i}u|||_{0,\omega_i}^2.
\end{equation*}
\end{lemma}

\section{GMsFEM space construction using local eigenvalue problems}\label{sec:gmsfem}

In this section, we summarize the construction of coarse scale finite element 
spaces using  a GMsFEM framework. In order to focus in the analysis of 
convergence we consider a particular case of the construction of 
spaces using the GMsFEM framework as introduced in 
\cite{egw10,EG09}. This construction evolved to 
the GMsFEM method as described in \cite{egh12}. The method
presented in this paper to obtain convergence can be 
also carried out for the constructions in \cite{egh12} under 
appropriate assumptions of the local spectral problems 
used for the construction of coarse spaces.

We choose the basis functions that  span
the eigenfunctions corresponding to small
eigenvalues.
We note that $\{\omega_{i}\}_{y_i\in \mathcal{T}^H}$ 
is a covering of $\Omega$. Let $\{\chi_i\}_{i=1}^{N_v}$ be a partition 
of unity subordinated to the covering
$\{\omega_{i}\}$ such that $\chi_i\in V^h(\Omega)$ and
$|\nabla \chi_i|\leq \frac{1}{H}$, $i=1,\dots,N_v$. Define the set of coarse 
basis functions
\begin{equation}\label{eq:def:Phiil}
\Phi_{i,\ell}=\chi_i\psi_\ell^{\omega_i} \quad \mbox{ for } 1\leq i \leq N_v
\mbox{ and } 1\leq \ell \leq L_i,
\end{equation}
where $L_i$ is the number of eigenvalues that will be chosen 
for the node $i$; see \cite{Babuska,Babuska2} for 
more details on the generalized finite element method 
using partitions of unity.
Denote by $V_0$, as before, the 
\emph{local spectral multiscale} space   
\begin{equation*}
V_N=\mbox{span}\{ \Phi_{i,\ell}: 1\leq  i \leq N_v
\mbox{ and } 1\leq \ell \leq L_i \}.
\end{equation*}
Define also 
\begin{equation*}
V_D=\mbox{span}\{ \phi_\ell^K:   K\in \mathcal{T}^H
\mbox{ and } 1\leq \ell \leq L_K \}.
\end{equation*}
Finally define, 
\begin{equation*}
V_0=V_N+V_D.
\end{equation*}

In practice, the computation of the multiscale basis functions have to be done. For instance, 
the computation of the multiscale basis functions can be performed in a fine grid (local 
to each region) that is sufficiently fine to resolve and represent the scales of the problem. 
In our case, this means that the fine-grid have to be sufficiently fine to represent 
the variations and discontinuities of the coefficient $\kappa$. The computations of the 
basis functions are local to each coarse region and can be done in a preprocessing 
step (taking advantage of parallel computations). For more details and related 
concepts, we refer to \cite{egh12}.

We define $u_H$ as the Galerkin approximation using the space $V_0$, that is,  
\begin{equation}\label{eq:coarseproblemMsFEM}
a(u_H,v)=f(v),\quad \text{for all}\  v\in V_0.
\end{equation}


%
%

\section{A technical assumption}

In order to get convergence rates we assume that we can decompose the exact solution $u$ as $u\approx u_D+u_N$
where the $u_D$ can be approximated using the space 
$V_D$ and $u_N$ can be approximated using the space $V_D$. 
This requires that $u_N$ and $u_D$ have the right boundary 
condition on the coarse block edges.
See Remark \ref{rem:nullnormalflux}. More precisely we estate the following assumption.

\begin{assumption}\label{assumption}
Let $u$ be the exact solution, that is 
$-\mbox{\normalfont div}(\kappa \nabla u)=f$. We assume that there exists 
$u_D$, $u_N$ and $\epsilon\preceq H$ such that 
\begin{enumerate}
    \item We have 
    \begin{equation}\label{eq:assA}
        \int_{\Omega}  \kappa |\nabla(u-u_D-u_N)|^2 \preceq \epsilon^2
        \int_{\Omega} \kappa^{-1}f^2 \preceq H^2 \int_{\Omega} \kappa^{-1} f^2
    \end{equation}
    \item We have the boundary data given by 
    \begin{equation*}
        \partial_\eta u_N =0 \mbox{ on } \partial K \mbox{ for all } K.
    \end{equation*}
    and 
    \begin{equation*}
         u_D =0 \mbox{ on } \partial K \mbox{ for all } K.
    \end{equation*}
    \item We have the bounds, 
    \begin{equation}\label{eq:assC}
        \int_{\Omega} \kappa^{-1}|\mbox{\normalfont  div}(\kappa \nabla u_N)|^2 \preceq
        \int_{\Omega} \kappa^{-1} |\mbox{\normalfont  div}(\kappa \nabla u)|^2 \mbox{ and }
        \int_{\Omega} \kappa^{-1} |\mbox{\normalfont  div}(\kappa \nabla u_D)|^2 \preceq
        \int_{\Omega} \kappa^{-1} |\mbox{\normalfont  div}(\kappa \nabla u)|^2
    \end{equation}
\end{enumerate}
\end{assumption}
Note that 2. in the case $\epsilon=0$ implies that $u_N=u$ on $\partial K$  and 
that $\partial_\eta u_D=\partial_\eta u$ 
on $\partial K$ and therefore  we should be able to split the Neumann 
and Dirichlet boundary data into different functions with regular divergece. This allows us to approximate each part by 
the rightly constructed subspace with eigenvalues with Nuemann and Dirichlet
data. We can think of Assumption \ref{assumption} as  a natural extension 
(to the case of variable coefficient) of a regularity assumption.  
If fact we can give a following example for the case of regular coefficient. 

\begin{remark}
In the case of regular coefficient $\kappa$ and regular right hand side $f$ it is known that, given any $\epsilon>0$ we can approximate the solution $u$ by $C^1$ finite elements defined on a sufficiently fine triangulation with square $H^1$-error smaller than 
$\epsilon^2 \int_{\Omega} f^2$ so we can get the  Assumption \ref{assumption}. In fact, 
for the case of  $C^1$ (Hermite) finite element spaces defined on a triangulation we can relate $u_D$ to the derivative value degrees 
of freedom while $u_N$ will correspond to the nodal values of the function $u$. 
\end{remark}

\begin{remark} In the case of regular coefficient, say 
$\kappa=1$, with regular right hand side it is known that the solution $u$ is regular $u\in H^2({\Omega})$. 
In this case, using standard regularity results we can show that 
Assumption \ref{assumption} holds with $\epsilon=0$.
We can construct $u_D$ and $u_N$ by solving a forth order problem. In fact, 
consider 
\[
\begin{array}{cl}
-\mbox{\normalfont div}(\kappa \nabla (\mbox{\normalfont  div} \kappa \nabla u_{K,D}))=0 &\mbox{ in } K,\\
\partial_\eta u_{K,D} = \partial_\eta u & \mbox{ on } \partial K,\\
 u_{K,D} = 0 & \mbox{ on } \partial K.\\
\end{array}
\]
Define the global function $u_D$ by $u_D|_K= u_{K,D}$. Note that $u_D\in H^1_0(\Omega)$.
Define also $u_N=u-u_D$. We have
\[
\begin{array}{cl}
\partial_\eta u_N= 0 & \mbox{ on } \partial K\\
 u_N = u & \mbox{ on } \partial K.\\
\end{array}
  \]
Note also that  $u_D+ u_N  =u$.
 
\end{remark}

\section{Approximation properties of the coarse space}\label{section8}
We mention that, in the presence of high-contrast multiscale coefficient $\kappa$, 
if $L$ is large enough then $\lambda_{L+1}^{\omega_i}$
is contrast independent and in this 
case we refer to 
(\ref{eq:truncation1})
as a contrast independent weighted Poincar\'e 
inequality.  The basis function encode information of the behavior of solutions due to the 
high-contrast in the multiscale coefficient and then allow us to compute using a 
coarse grid size $H$that does not need to resolve all discontinuities of the coefficient 
$\kappa$. 
This is a main motivation for the construction of the space presented 
above. For many recent developments using these ideas we refer 
the interested reader to \cite{egh12} and references there in. 
In this paper, in order to focus in the convergence analysis, we work with 
the coarse space presented above. We also note that, more 
involved and sophisticated coarse space can be constructed as in 
\cite{egh12}. The ideas developed in this paper also apply to 
variety of cases proposed in \cite{egh12}.

\subsection{A coarse-scale interpolation 
operator}

Given an integer $L$, and $v\in V^h(\Omega)$, we define 
\begin{equation}\label{eq:def:I-Omega-L2}
I^{\omega_i}_{L}v=\sum_{\ell=1}^{L} 
 m^{\omega_i}(v,\psi_\ell^{\omega_i})\psi_\ell^{\omega_i}. 
\end{equation}
From Lemma \ref{lem:aprioriestN} the following inequality holds for 
$1\leq  s\leq 2$,
\begin{equation}\label{eq:truncation2}
\int_{\omega_i} \kappa(v-I_L^{\omega_i} v)^2\leq 
\frac{1}{(\lambda_{L+1}^{\omega_i})^{s}} |||v-I_L^{\omega_i} v|||_{s,\omega_i}^2
\leq 
\frac{1}{(\lambda_{L+1}^{\omega_i})^{s}} |||v|||_{s,\omega_i}^2
\end{equation}
and, if $\partial_\eta v=0$ on $\partial \omega_i$,  we have 

\begin{equation}\label{eq:truncationenergy}
\int_{\omega_i} \kappa |\nabla (v-I_L^{\omega_i} v)|^2\leq 
\frac{1}{(\lambda_{L+1}^{\omega_i})^{(s-1)}} |||v-I_L^{\omega_i} v|||_{s,\omega_i}^2
\leq 
\frac{1}{(\lambda_{L+1}^{\omega_i})^{(s-1)}} |||v|||_{s,\omega_i}^2.
\end{equation}

Recall that we assume that 
$\partial_\eta v=0$ on $\partial \omega_i$, otherwise the  
term $\sum_{\ell=1}^\infty \int_{\partial \omega_i} \partial_\eta v \psi^{\omega_i}_\ell$ will come on the right hand side. These last term is harder to bound.\\

Define  the coarse interpolation $I_N$ by
\begin{align*}
I_N v  &=\sum_{i=1}^{N_c} \sum_{\ell=1}^{L_i}
\left(\int_{\omega_i}\kappa v
\psi_\ell^{\omega_i} \right)
\chi_i\psi_{\ell}^{\omega_i}= \sum_{i=1}^{N_c}  (I^{\omega_i}_{L_i}u_N)\chi_i
\end{align*}
where $I^{\omega_i}_{L_i}$  is defined in 
(\ref{eq:def:I-Omega-L2}).  
 Note that we have
\[
v-I_N v=\sum_{i=1}^{N_c} 
\chi_i(v-I^{\omega_i}_{L_i}v).
\]

This interpolation was analyzed in  \cite{ge09_1,ge09_1reduceddim} 
for high-contrast problems, there, it was used to obtain 
a robust two level domain decomposition method and no approximation 
in energy norm was needed. 
Later, in \cite{egw10} an analysis to obtain approximation results in $H^1$ norm was 
carried out. The analysis was rather 
complex and difficult to extend to other applications. The analysis 
we present in this paper simplifies that 
of \cite{egw10}. Indeed and it is easier to extend and to combine
with different techniques to 
analyze different realizations of the GMsFEM methodology.

In order to avoid the assumption of square integrable residuals in 
the approximation result 
as it is done in \cite{egw10} (which implies assuming that the solution 
has zero flux across neighborhood boundaries - see Remark
\ref{rem:nullnormalflux}) we  introduce an additional interpolation 
operator into $V_D$. 
Given $L_K$ and $v$ define 
\begin{equation*}
J^{K}v=\sum_{\ell=1}^{L_K} 
 m^{K}(v,\phi_\ell^{K})\phi_\ell^{K}. 
\end{equation*}
and
\begin{equation*}
    J_D v = \sum_{K\in \mathcal{T}^H} J^{K} v
\end{equation*}
where we extend $J_{L_K}^{K} v$ by zero outside the block $K_j$.

\subsection{Interpolation approximation}

We note that the analysis presented in this 
section is closely related to the analysis in \cite{ge09_1,ge09_1reduceddim, egw10}. 
In particular we simplify and the analysis presented in \cite{egw10}.

\begin{lemma}
Consider $v\in H^1_0(\Omega)$. We have the following weighted $L^2$ approximation
\begin{equation}\label{eq:I0approx}
\int_{K}\kappa  (v-I_Nv)^2\preceq 
\frac{1}{\lambda_{K,L+1}}
\sum_{y_i\in K}||v-I_{L_i}^{\omega_i}v||_{H^1(\omega_i)}^2\end{equation}
where $\lambda_{K,L+1}=\min_{y_i\in K}\lambda_{L_i+1}^{\omega_i}$ and therefore $\int_{\Omega}\kappa  (v-I_0v)^2\preceq 
\frac{1}{\lambda_{L+1}}
||v||_{H^1(\Omega)}^2$ where $\lambda_{L+1}=\min_{K} \lambda_{K,L+1}$.
\end{lemma}

\begin{proof}
First we prove (\ref{eq:I0approx}). 
Using 
 that 
$\chi_i\leq 1$ we have 
\begin{eqnarray*}
\int_{K}\kappa  (v-I_0v)^2&\preceq& 
 \sum_{y_i\in K}
\int_{K} \kappa (\chi_i(v-I^{\omega_i}_{L_i}v))^2 \\\
&\preceq&  \sum_{y_i\in K}
\int_{\omega_i} \kappa (v-I^{\omega_i}_{L_i}v)^2
\end{eqnarray*}
and using (\ref{eq:truncation2}) 
to estimate the last term above, 
we obtain the result.
\end{proof}

We now present the result in the $H^1$-norm. Here we are more explicit in the assumption than in  
the analogous result in \cite{egw10} where they assume that in each neighborhood the residual is square integrable. See \cite{egw10}. The proof is analogous to the one presented in \cite{egw10} and we presented here for completeness. 

\begin{lemma}\label{lem:coarse-projection2}
Assume that $f=\mbox{\normalfont div}(\kappa \nabla u)\in L^2(w^K)$ and also assume that 
for each $i$, $y_i\in K$ we have $\partial_\eta u=0$ on $\partial \omega_i$. 
Then, the following energy approximation holds,
\begin{equation}
\int_{K}\kappa  |\nabla u-\nabla I_N u|^2\preceq 
\max\left\{\frac{1}{H^2\lambda_{K,L+1}^{2}},\frac{1}{\lambda_{K,L+1}}\right\}
||f||_{L^2(w^K)}^2
\label{eq:Iomegah1approx}
\end{equation} 
where $\lambda_{K,L+1}=\min_{y_i\in K}\lambda_{L_i+1}^{\omega_i}$ and 
$\omega^K$ is defined in 
\eqref{eq:def:omegaK}.
\end{lemma}

\begin{proof}
We note that 
$\sum_{y_i\in K} \nabla \chi_i=0$ in $K$,
and  then we can 
fix $y_j\in K$ and
write $\nabla\chi_j=
-\sum_{y_i\in K\setminus\{y_j\}} \nabla \chi_i$. 
We obtain,
\begin{eqnarray*}
\nabla  \sum_{y_i\in K}  (v-I^{\omega_i}_{L_i}v)\chi_i
&=&\sum_{y_i\in K}  
\nabla \chi_i  (v-I^{\omega_i}_{L_i}v) +
\sum_{y_i\in K} \chi_i  \nabla (v-I^{\omega_i}_{L_i}v) \\
&=&\sum_{y_i\in K\setminus\{y_j\}}  
(I^{\omega_i}_{L_i}v-I^{\omega_j}_{L_j}v)\nabla\chi_i  +
\sum_{y_i\in K} \chi_i\nabla(v-I^{\omega_i}_{L_i}v) 
\end{eqnarray*}
which gives the following bound valid on $K$, 
\begin{eqnarray}
&&|\nabla  \sum_{y_i\in K}  (v-I^{\omega_i}_{L_i}v)\chi_i|^2
\preceq \frac{1}{H^2}\sum_{y_i\in K\setminus\{y_j\}}  
(I^{\omega_i}_{L_i}v-I^{\omega_j}_{L_j}v)^2+
\sum_{y_i\in K} |\nabla(v-I^{\omega_i}_{L_i}v)|^2.\label{eq:gradsumI0local}
\end{eqnarray}
From  (\ref{eq:gradsumI0local}) we get
\begin{eqnarray}
&&\int_{K}\kappa |\nabla (v-I_0v)|^2 
\preceq \int_{K} \kappa 
|\nabla  \sum_{y_i\in K}  (v-I^{\omega_i}_{L_i}v)\chi_i|^2
\nonumber\\ 
&\preceq&
\sum_{y_i\in K} 
\frac{1}{H^2}\int_{K} \kappa (I^{\omega_i}_{L_i}v-I^{\omega_j}_{L_j}v)^2 
+\sum_{y_i\in K}\int_{K} \kappa |\nabla(v-I^{\omega_i}_{L_i}v)|^2.
\label{eq:twotermsgradI0}
\end{eqnarray}

To bound the first term above we use 
(\ref{eq:truncation2}) as follows, 
\begin{eqnarray}
&&\int_{K} \kappa (I^{\omega_i}_{L_i}v-I^{\omega_j}_{L_j}v)^2 
\preceq 
\int_{\omega_i} \kappa (v-I^{\omega_i}_{L_i}v)^2+
\int_{\omega_j} \kappa (v-I^{\omega_i}_{L_i}v)^2\nonumber\\
&\preceq & 
\frac{1}{(\lambda_{L+1}^{\omega_i})^{2}} |||v-I_{L_i}^{\omega_i} v|||_{2,\omega_i}^2+
\frac{1}{(\lambda_{L+1}^{\omega_j})^{2}} |||v-I_{L_j}^{\omega_j} v|||_{2,\omega_j}^2\nonumber\\
&\preceq &
\frac{1}{(\lambda_{K,L+1})^2} \sum_{y_i\in K}|||v-I_{L_i}^{\omega_i} v|||_{2,\omega_i}^2. \label{eq:firstgradI0}
\end{eqnarray}
The second term in (\ref{eq:twotermsgradI0}) 
is estimated using (\ref{eq:truncationenergy}) 
\begin{eqnarray}
\int_{K} \kappa |\nabla(v-I^{\omega_i}_{L_i}v)|^2 \leq
\int_{\omega_i} \kappa |\nabla(v-I^{\omega_i}_{L_i}v)|^2
\preceq\frac{1}{\lambda_{L+1}^{\omega_i}} |||v-I_{L_i}^{\omega_i} v|||_{2,\omega_i}^2.\label{eq:secondgradI0}
\end{eqnarray}
By combining  (\ref{eq:firstgradI0}), (\ref{eq:secondgradI0})
and (\ref{eq:twotermsgradI0}) we obtain (\ref{eq:Iomegah1approx}).
\end{proof}

A similar lemma for the case of 
Dirichlet boundary conditions on 
$K$ is presented next. This is 
direct consequence of Lemma 
\ref{lem:aprioriestD}.
\begin{lemma}\label{lem:coarse-projection3}
Assume that $f=\mbox{\normalfont div}(\kappa \nabla u)\in L^2(\Omega)$ and also assume that 
for each $K \in \mathcal{T}^H$,  we have $u=0$ on $\partial K$. 
Then, the following energy approximation holds,
\begin{equation*}
\int_{K}\kappa  |\nabla u-\nabla J_D u|^2 \preceq 
\frac{1}{\mu^{K}_{L+1}}
||f||_{L^2(K)}^2.
\end{equation*} 
\end{lemma}

The idea is to use the Assumption \ref{assumption} and then to apply $I_N$ to $u_N$ and the other part, that is $u_D$ will be approximated by a truncated expansion on $V_D$. Under Assumption \ref{assumption} and for the solution $u$ define  the coarse interpolation $I_0$ by
\begin{align}
I_0u  &=\sum_{i=1}^{N_c} \sum_{\ell=1}^{L_i}
\left(\int_{\omega_i}\kappa u_N
\psi_\ell^{\omega_i} \right)
\chi_i\psi_{\ell}^{\omega_i}  +  
\sum_{K\in\mathcal{T}^H} \sum_{\ell=1}^{L_K}
\left(\int_{K}\kappa u_D
\phi_\ell^{K} \right)
\phi_{\ell}^{K}\nonumber\\
    &= \sum_{i=1}^{N_c}  (I^{\omega_i}_{L_i}u_N)\chi_i +
    \sum_{K\in\mathcal{T}^H}  (J^{K}u_D) =
    I_N u_N + J_D u_D. \label{eq:def:I0}
\end{align}
Recall that $L_i$ is the number of Newmann eigenfunctions considered in the neighborhood $w_i$ and $L_K$ is the number of Dirichlet eigenfunctions considered on the element $K$. We finally present our main approximation result.
\begin{theorem}\label{lem:coarse-projection-global}
 Assume that $|||u|||_{2,\Omega}<\infty $ 
where $u$  is the solution of ($\ref{eq:problem1}$)
and that Assumption \ref{assumption} holds, 
the following approximation for the energy interpolation error holds,

\begin{equation}\label{eq:I0stabGLOBAL}
\int_{\Omega}\kappa  |\nabla u-\nabla I_0u |^2\preceq 
\left(\max\left\{\frac{1}{H^2(\lambda_{L+1})^{2}},\frac{1}{\lambda_{L+1}}
\right\} +
\frac{1}{\mu_{L+1}}+\epsilon^2
\right)
|||\kappa^{-1/2}f|||_{0}^2
\end{equation}

where $\lambda_{L+1}=\min_{ K}\lambda_{K,L+1}^{\omega_i}$ 
and 
$\mu_{L+1}=\min_{ K}\mu_{L+1}^{K}$.
\end{theorem}

\begin{proof}
By definition (\ref{eq:def:I0}), our technical assumption and using the triangular inequality we have 
\begin{align}
 \int_{\Omega}\kappa  |\nabla u-\nabla I_0u |^2 & = \int_{\Omega}\kappa  |\nabla (u + u_N + u_D - (u_N + u_D))-\nabla ( I_N u_N + J_D u_D) |^2 \nonumber\\
 &\preceq \int_{\Omega}\kappa  |\nabla (u_N - I_N u_N )|^2  + \int_{\Omega}\kappa  |\nabla (u_D  - J_D u_D) |^2 + \int_{\Omega}\kappa  |\nabla (u - u_N - u_D)|^2.\label{eq:threeSum}
\end{align}
Using Lemma \ref{lem:coarse-projection2} and noting that $\bigcup K \subset \bigcup w^K$ we have for the firs term in (\ref{eq:threeSum})
\begin{align}
    \int_{\Omega}\kappa  |\nabla (u_N - I_N u_N )|^2 &= \sum_{K\in\mathcal{T}^H}\int_{K}\kappa  |\nabla (u_N - I_N u_N )|^2\nonumber\\ 
    &\leq \sum_{w^K}\max\left\{\frac{1}{H^2\lambda_{K,L+1}^{2}},\frac{1}{\lambda_{K,L+1}}\right\} ||k^{-1}\div(\kappa \nabla u_N)||_{L^2(w^K)}^2\nonumber\\
    &\preceq  \max_{K\in\mathcal{T}^H}\left\{\max\left\{\frac{1}{H^2\lambda_{K,L+1}^{2}},\frac{1}{\lambda_{K,L+1}}\right\}\right\} \sum_{w^K} ||k^{-1}\div(\kappa \nabla u_N)||_{L^2(w^K)}^2\nonumber\\
    &\preceq \max\left\{\frac{1}{H^2\lambda_{L+1}^{2}},
    \frac{1}{\lambda_{L+1}}\right\} ||k^{-1}\div(\kappa \nabla u_N)||_{L^2(\Omega)}^2\label{eq:sumando1}
\end{align}
Now using Lemma \ref{lem:coarse-projection3} for the second term in (\ref{eq:threeSum}) we get
\begin{align}
   \int_{\Omega}\kappa  |\nabla (u_D  - J_D u_D) |^2 &=\sum_{K\in\mathcal{T}^H}\int_{K}\kappa  |\nabla (u_D  - J_D u_D) |^2 \preceq 
   \sum_{K\in\mathcal{T}^H}\frac{1}{\mu^{K}_{L+1}} ||k^{-1}\div(\kappa \nabla u_D)||_{L^2(K)}^2\nonumber\\
   &\preceq \max_{K\in\mathcal{T}^H}\left\{\frac{1}{\mu_{K,L+1}}\right\} \sum_{K\in\mathcal{T}^H}||k^{-1}\div(\kappa \nabla u_D)||_{L^2(K)}^2\nonumber\\
   &\preceq  \frac{1}{\mu_{L+1}} ||k^{-1}\div(\kappa \nabla u_D)||_{L^2(\Omega)}^2\label{eq:sumando2}
\end{align}
using (\ref{eq:assA}) we obtain the bounds for the third term in (\ref{eq:threeSum}) 
\begin{equation}\label{eq:thirdsum}
    \int_{\Omega}\kappa  |\nabla (u - u_N - u_D)|^2 \preceq \epsilon^2
        \int_{\Omega} \kappa^{-1}f^2
\end{equation}
with (\ref{eq:thirdsum}) and using (\ref{eq:assC}) in (\ref{eq:sumando1}) and (\ref{eq:sumando2})  we obtain (\ref{eq:I0stabGLOBAL}) from (\ref{eq:threeSum}).
\end{proof}

With the tools we have at hand we can obtain the convergence of the GMsFEM 
method of Section \ref{sec:gmsfem}. Under 
Assumption \ref{assumption}, 
by combining the Cea's lemma with our interpolation approximation result 
(Lemma \ref{lem:coarse-projection-global} )
and the estimates in Lemma \ref{lem:aprioriest} we obtain
the following error estimates.

\begin{theorem}\label{thm:finalerror}
Let $u$ be the solution of problem (\ref{eq:problem}) with $f$ being 
square integrable and let $u_H$ the solution of (\ref{eq:coarseproblemMsFEM})
using the coarse basis functions constructed in (\ref{eq:def:Phiil}). 
Suppose  Assumption \ref{assumption} holds. We have,
\[
\int_{\Omega}\kappa|\nabla(u-u_H)|^2\preceq 
\left(\max\left\{\frac{1}{H^2(\lambda_{L+1})^{2}},\frac{1}{\lambda_{L+1}}
\right\} +
\frac{1}{\mu_{L+1}}+\epsilon^2
\right)
\int_{\Omega} \kappa^{-1}
f^2
\]
where ${\lambda_{L+1}}$ is the minimum left-out eigenvalue and was introduced in Theorem  \ref{lem:coarse-projection-global}
\end{theorem}

It is easy to see that if we map the local eigenvalue problem posed in 
$\omega_i$ (of diameter $H$) to a size one domain, then, the resulting 
eigenvalues scale with $H^{-2}$. If we re-scale all the eigenvalue problem 
to size one domains, we can then write the estimates in terms of 
eigenvalue problems posed in one size domains. In this way it is clear 
the $H$ dependence of our estimate.
We have the following result.
\begin{corollary}
Under the assumptions of Theorem \ref{thm:finalerror} we have,
\[
\int_{\Omega}\kappa|\nabla(u-u_H)|^2\preceq H^{2} \int_{\Omega} \kappa^{-1}
f^2,
\]
where the hidden constant involves eigenvalues of  re-scaled eigenvalue problems posed on the unit square.
\end{corollary}


\section{Numerical experiments}\label{sec:numerics}

Our aim is to show that, when applying GMsFEM, including Dirichlet's basis functions 
is necessary in some cases. So let's consider problem (\ref{eq:problem1}) on a 
square domain $\Omega$ subject to homogeneous Dirichlet's boundary conditions. Define a fine 
rectangular $512 \times 512$ mesh $\mathcal{T}^h$ and a coarse rectangular 
$4 \times 4$ mesh $\mathcal{T}^H$. Coarse basis functions are associated to 
each nodal point of $\mathcal{T}^H$ and supported on it's $4$ adjacent rectangles. We  apply the GMsFEM method to problem (\ref{eq:problem1}) with homogeneous  medium ($\kappa \equiv 1$) considering $2$ different 
smooth sources. \\

\textbf{Experiment $1$}: Apply GMsFEM to problem (\ref{eq:problem1}) with source $f_1$ being a smooth function composed by product of polynomials and sines which correspond to the exact solution $u_1$ 
\begin{equation}
\begin{cases}
    u_1(x,y) = \sin(\pi x)    \sin(\pi y)  (-x+3y)\\
    f_1(x,y) = 2\pi \cos(\pi x)\sin(\pi y) - 6 \pi \sin(\pi x) \cos(\pi y) + 2 \pi^2  \sin(\pi x) \sin(\pi y)   (-x+3y)
    \label{Exemplo1}
\end{cases}
\end{equation}

\textbf{Experiment $2$}: Apply GMsFEM to problem (\ref{eq:problem1}) with source $f_2$ being a linear combination of sines that correspond to the exact solution $u_2$ also constructed by the linear combination  of the $\mathbb{R}^2$ tonsorial product of functions $\{ \sin(4 k \pi ) \}_{k = 1,2}$
\begin{equation}\label{Exemplo2}
   \begin{cases}
      u_2(x,y) =  \sum_{k,l=1}^2 c_{k,l} ~  \sin(4k\pi) \sin(4l\pi)\\
      f_2(x,y) =  \sum_{k,l=1}^2 16(k+l) ~ c_{k,l} ~ \sin(4k\pi)  \sin(4l\pi)
   \end{cases}
\end{equation}
with $c_{i j} = 1/(i + 2(j-1))$ decreasing as $i$ and $j$ grow. \\

We vary the number  of Newman function  $N_N = 1 \cdots 40$ while fixing the number of Dirichlet function by element  ($N_D = 0 \cdots 20$ in Experiment $1$ and $N_D = 0 \cdots 3$ in Experiment $2$), . The errors in the energy norm are presented in Figure \ref{EnergyErrors1} while in Figures \ref{Smooth1} and \ref{Smooth2} we can see the evolution of solution $u_1$ and $u_2$  respectively.\\

\begin{figure}[h!]
  \begin{minipage}[b]{0.5\linewidth}
  		\centering
  		\includegraphics[scale=.5]{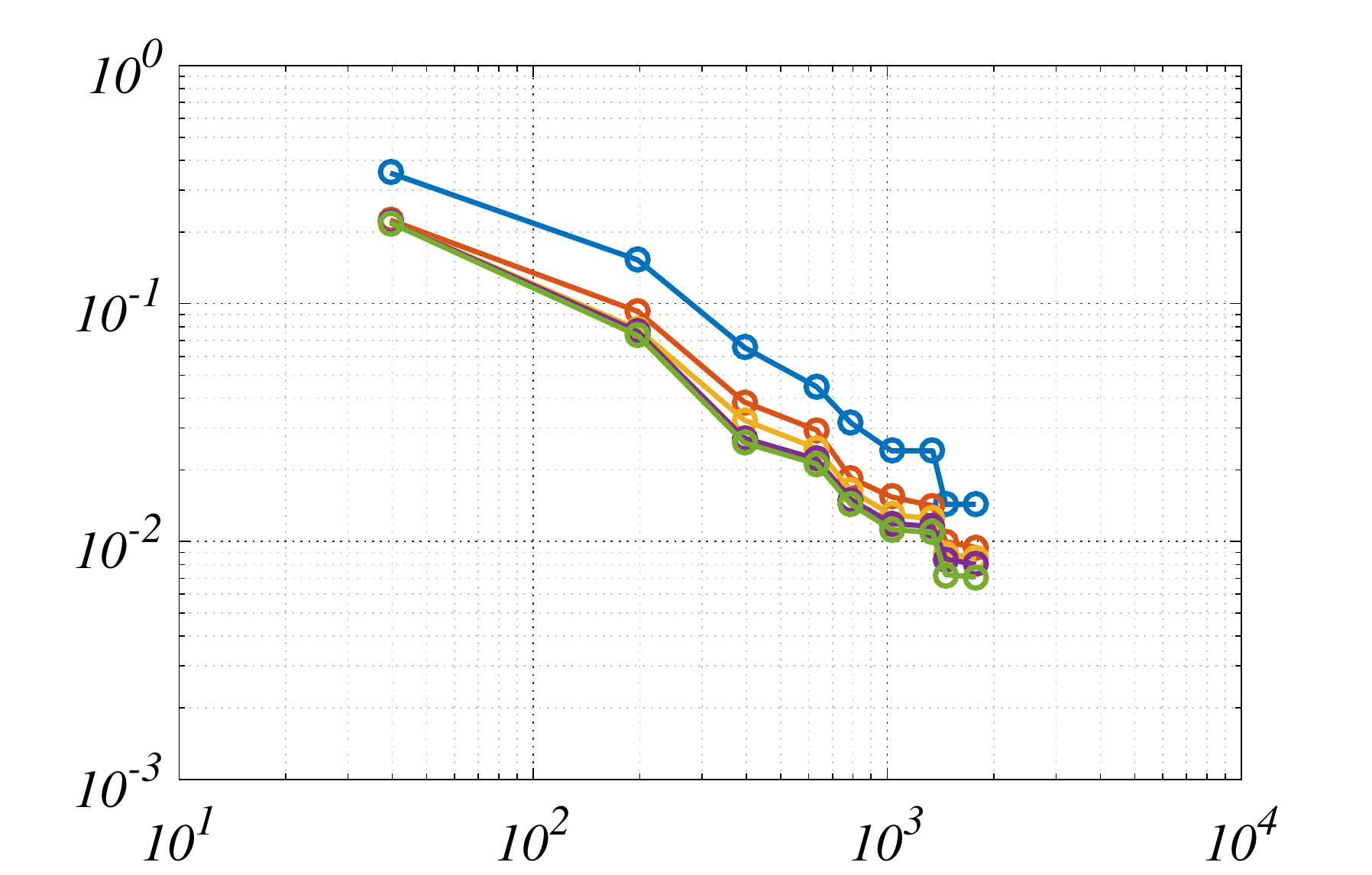}
  \end{minipage}  
  \begin{minipage}[b]{0.5\linewidth}
  		\centering
        \includegraphics[scale=.5]{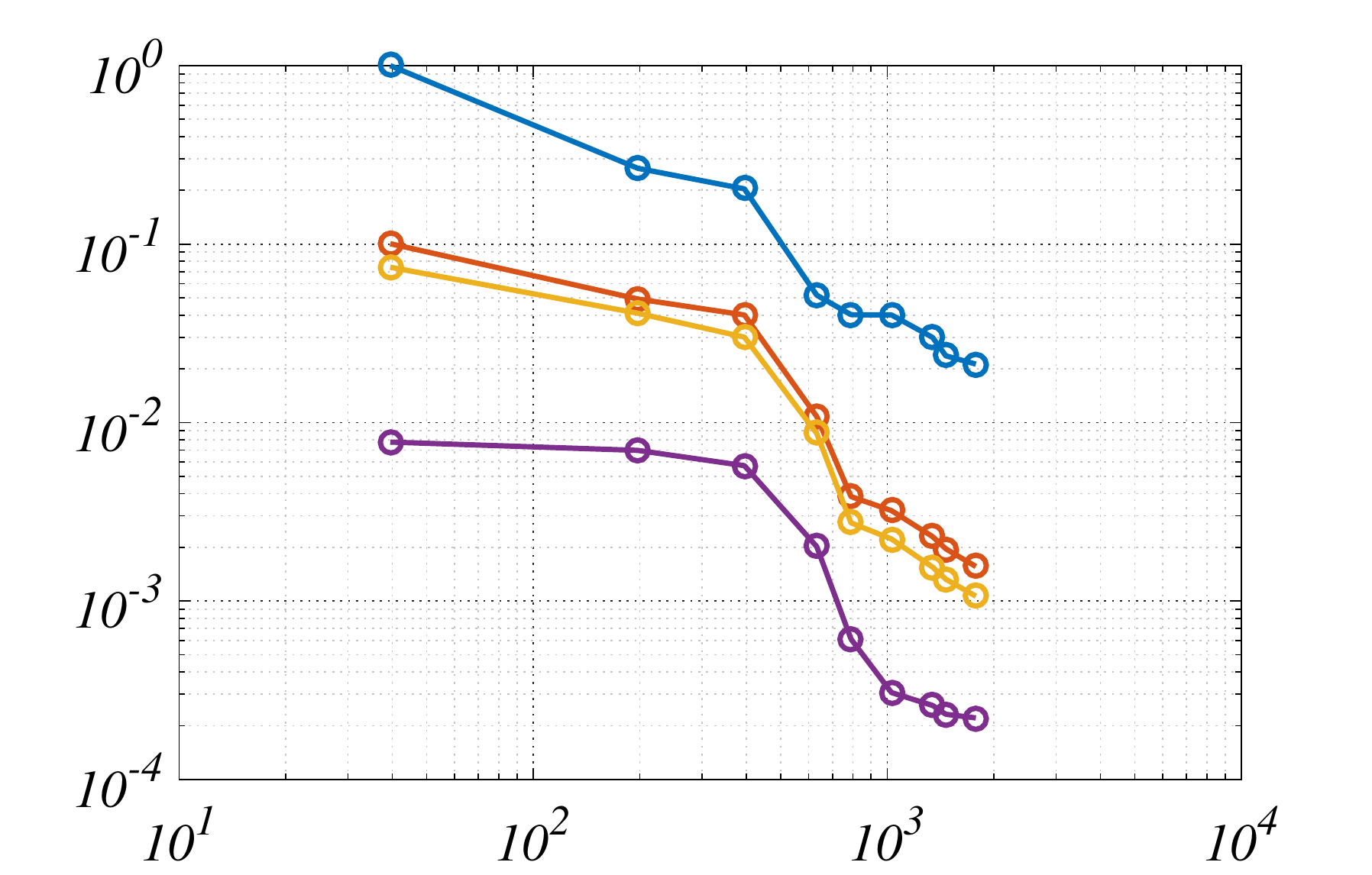}
  \end{minipage}  
  \caption{Log-Log plots of errors in the energy norm  vs. first eigenvalue out of the expansion after applying GMsFEM to Examples $1$ and $2$. (Left)  source term and exact solution in (\ref{Exemplo1}). (Right) source term and exact solution in  (\ref{Exemplo2}).  Each color correspond to a fixed  $N_D = \{ 0,5,10,15,20\} $ and ($ N_D = \{ 0,1,2,3\} $ (right)) and each circle correspond to $N_N = \{1,5,10,...,40\}$.}
  \label{EnergyErrors1}
\end{figure}

\begin{figure}[h!]
  \begin{minipage}[b]{0.5\linewidth}
  		\centering
  		\includegraphics[scale=.4]{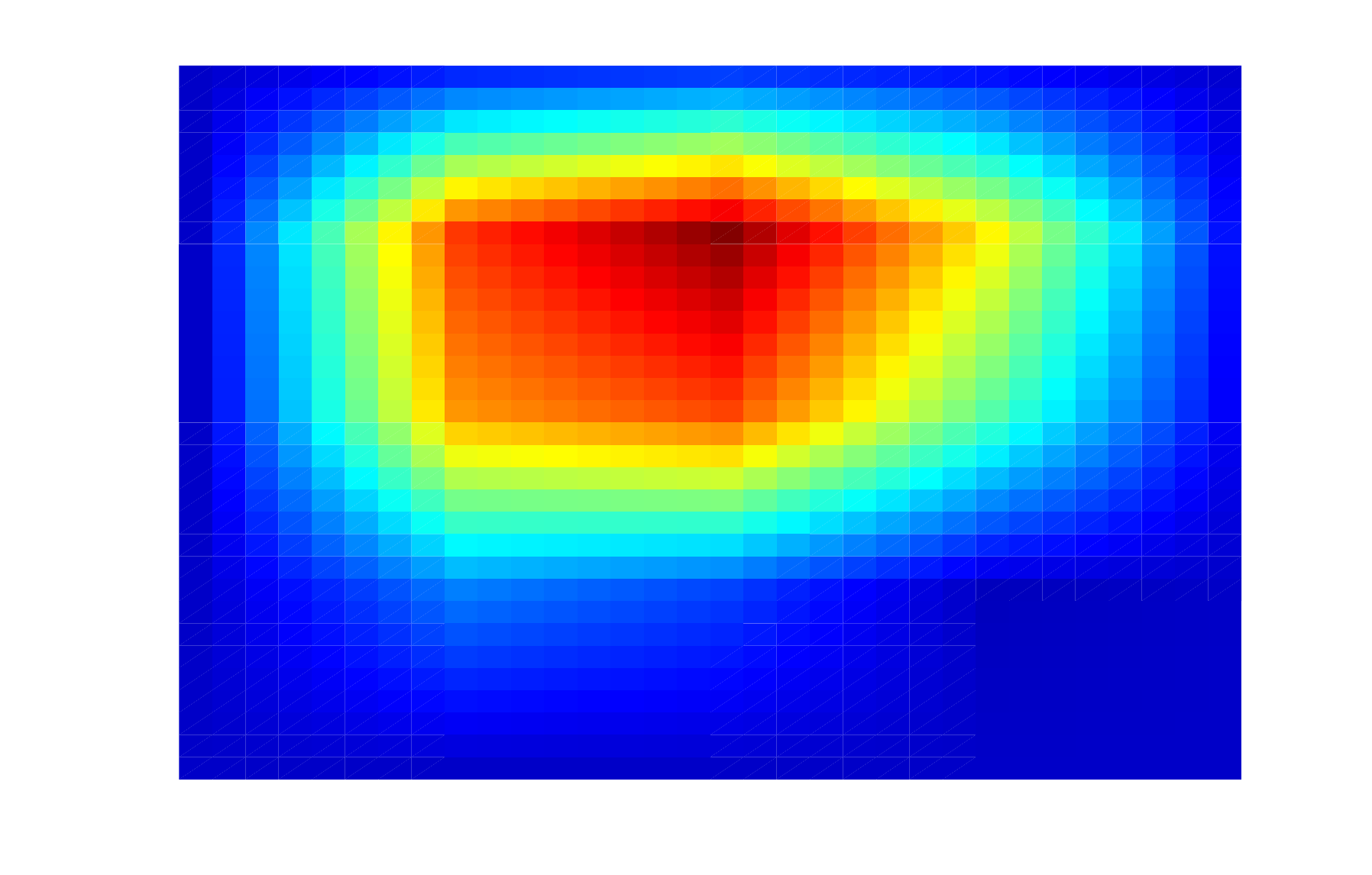}
        \captionsetup{labelformat=empty}
  \end{minipage}  
  \begin{minipage}[b]{0.5\linewidth}
  		\centering
        \includegraphics[scale=.4]{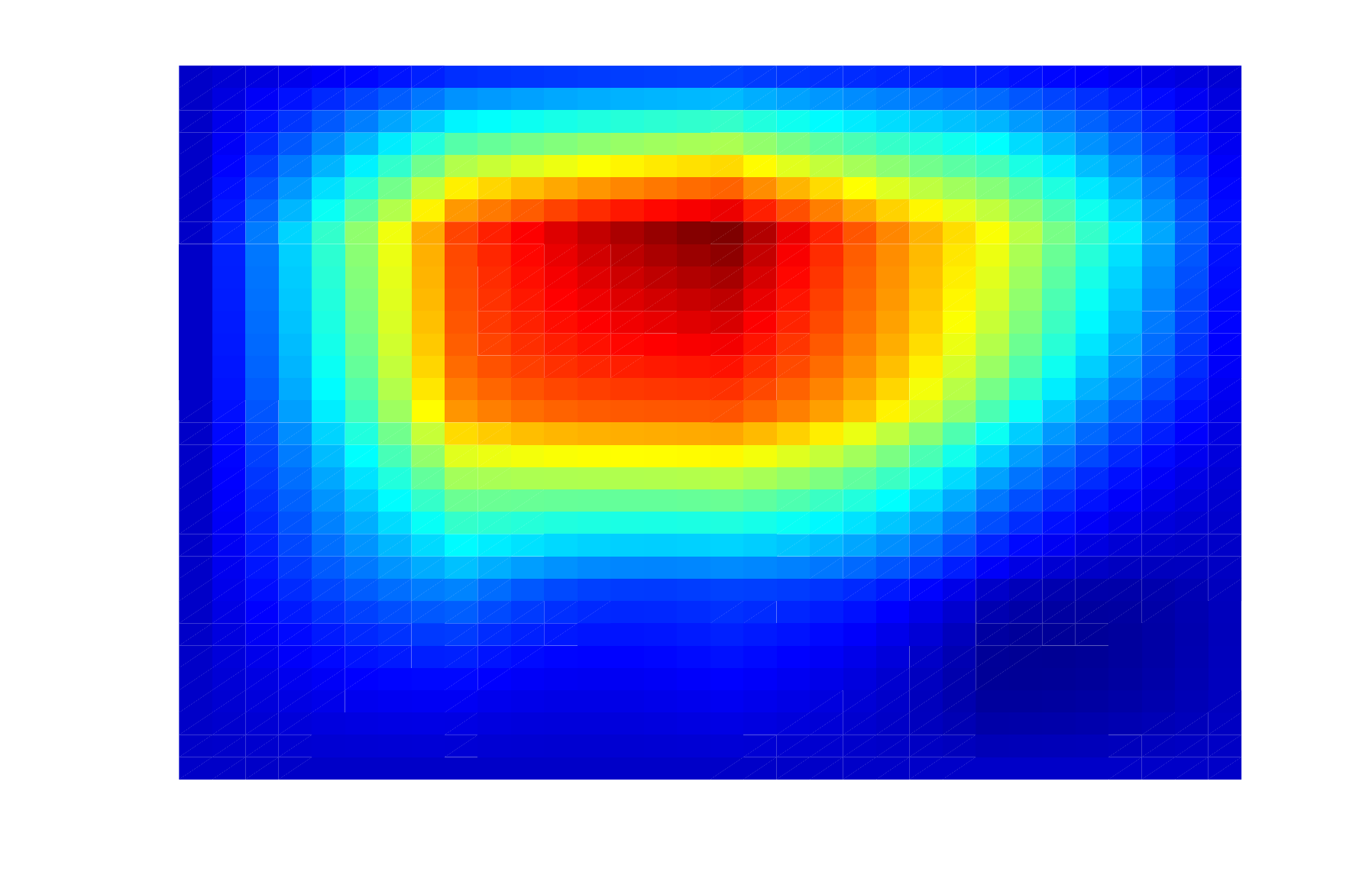}
        \captionsetup{labelformat=empty}
  \end{minipage}  
  \begin{minipage}[b]{0.5\linewidth}
  		\centering
  		\includegraphics[scale=.4]{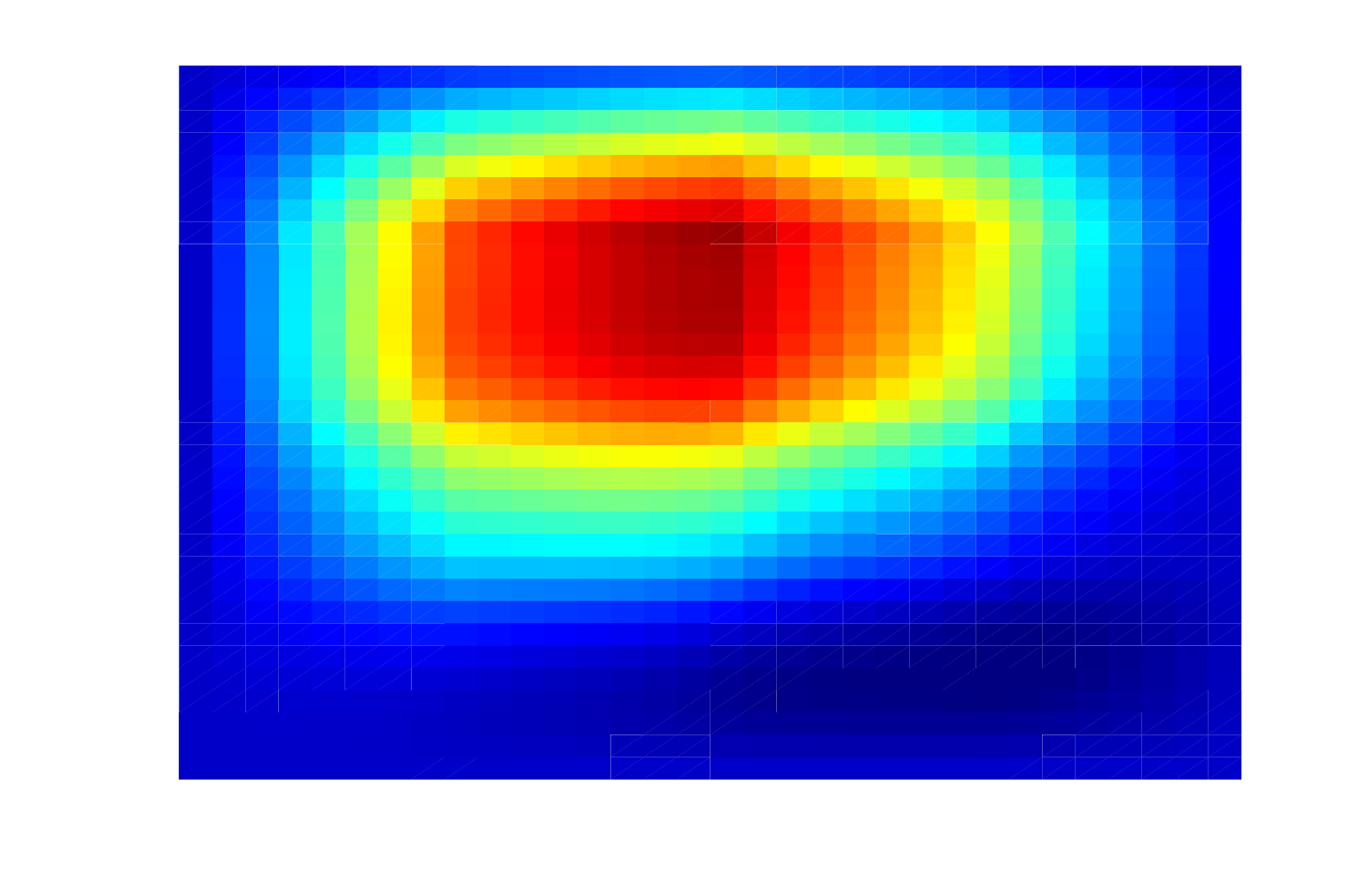}
        \captionsetup{labelformat=empty}
  \end{minipage} 
  \begin{minipage}[b]{0.5\linewidth}
		\centering
        \includegraphics[scale=.4]{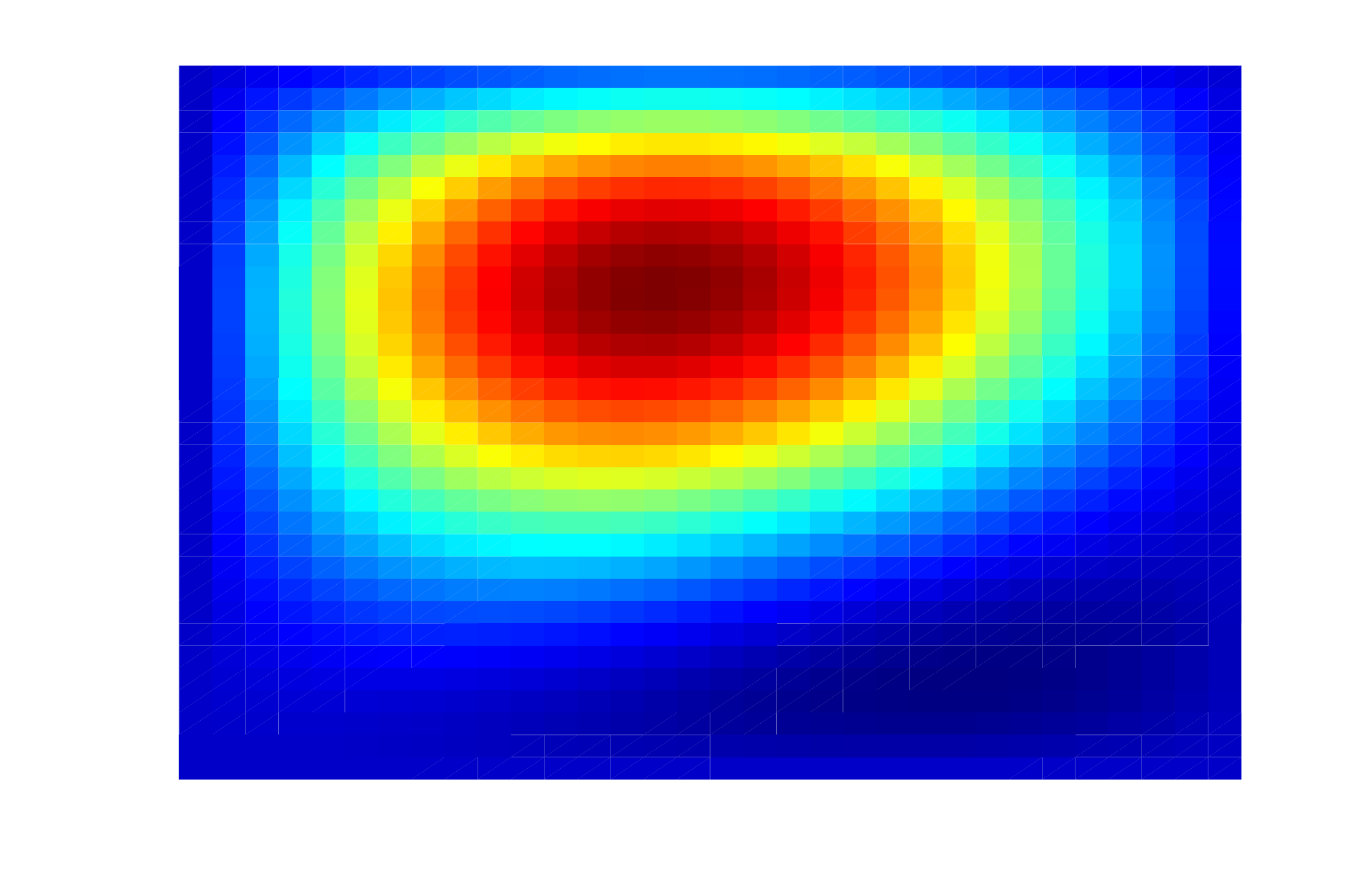}
        \captionsetup{labelformat=empty}
  \end{minipage}
  \caption{Level plots of $3$ approximations of exact solution $u_1$ in (\ref{Exemplo1}). (Up-Left 2-$N_N$, 2-$N_D$), (Up-Right 3-$N_N$, 3-$N_D$) and (Down-Left 4-$N_N$, 4-$N_D$) evolving to exact solution (Down-Right).}
\label{Smooth1}
\end{figure}

\begin{figure}[h!]
  \begin{minipage}[b]{0.5\linewidth}
  		\centering
  		\includegraphics[scale=.4]{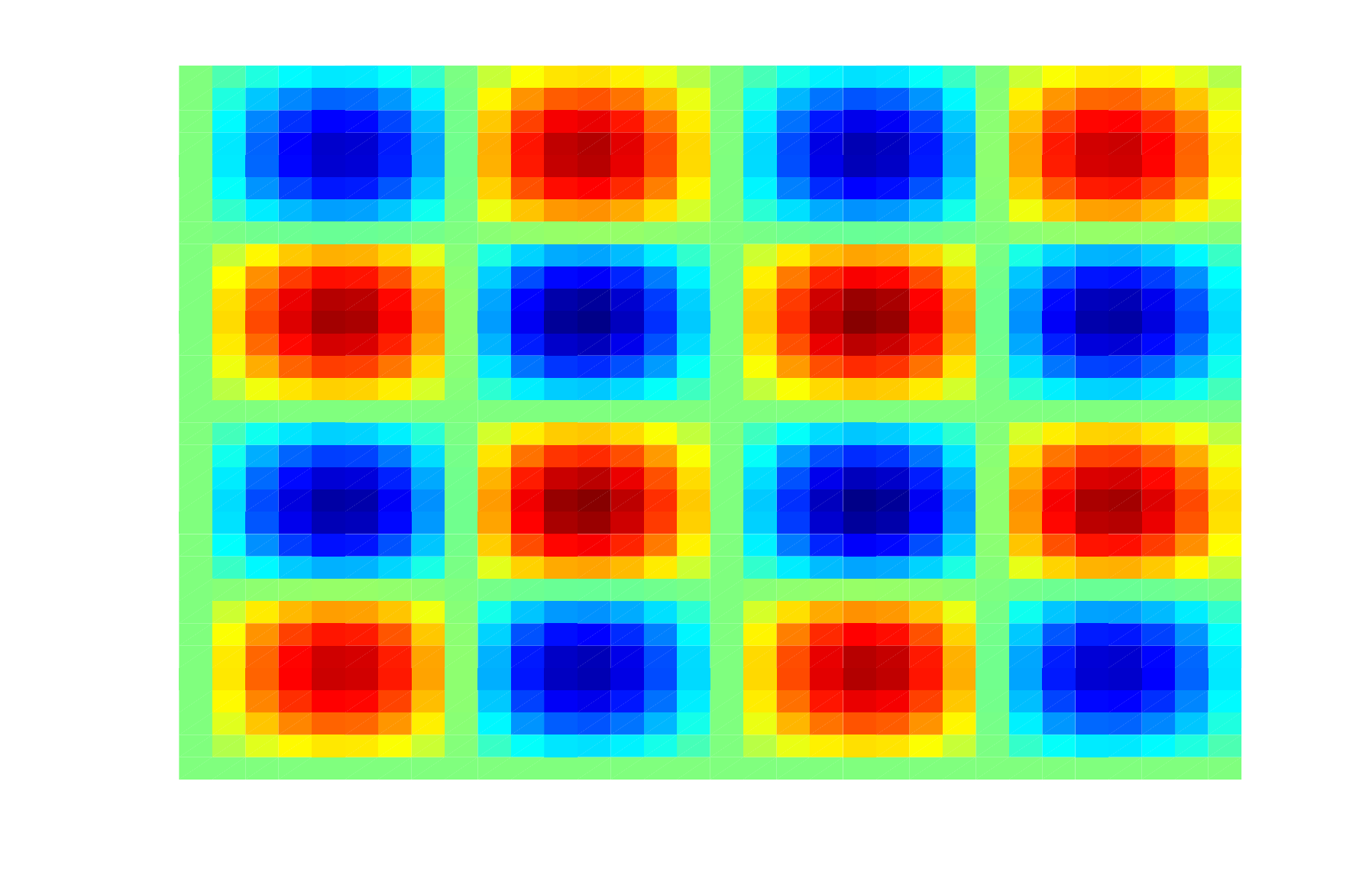}
        \captionsetup{labelformat=empty}
  \end{minipage}  
  \begin{minipage}[b]{0.5\linewidth}
  		\centering
        \includegraphics[scale=.4]{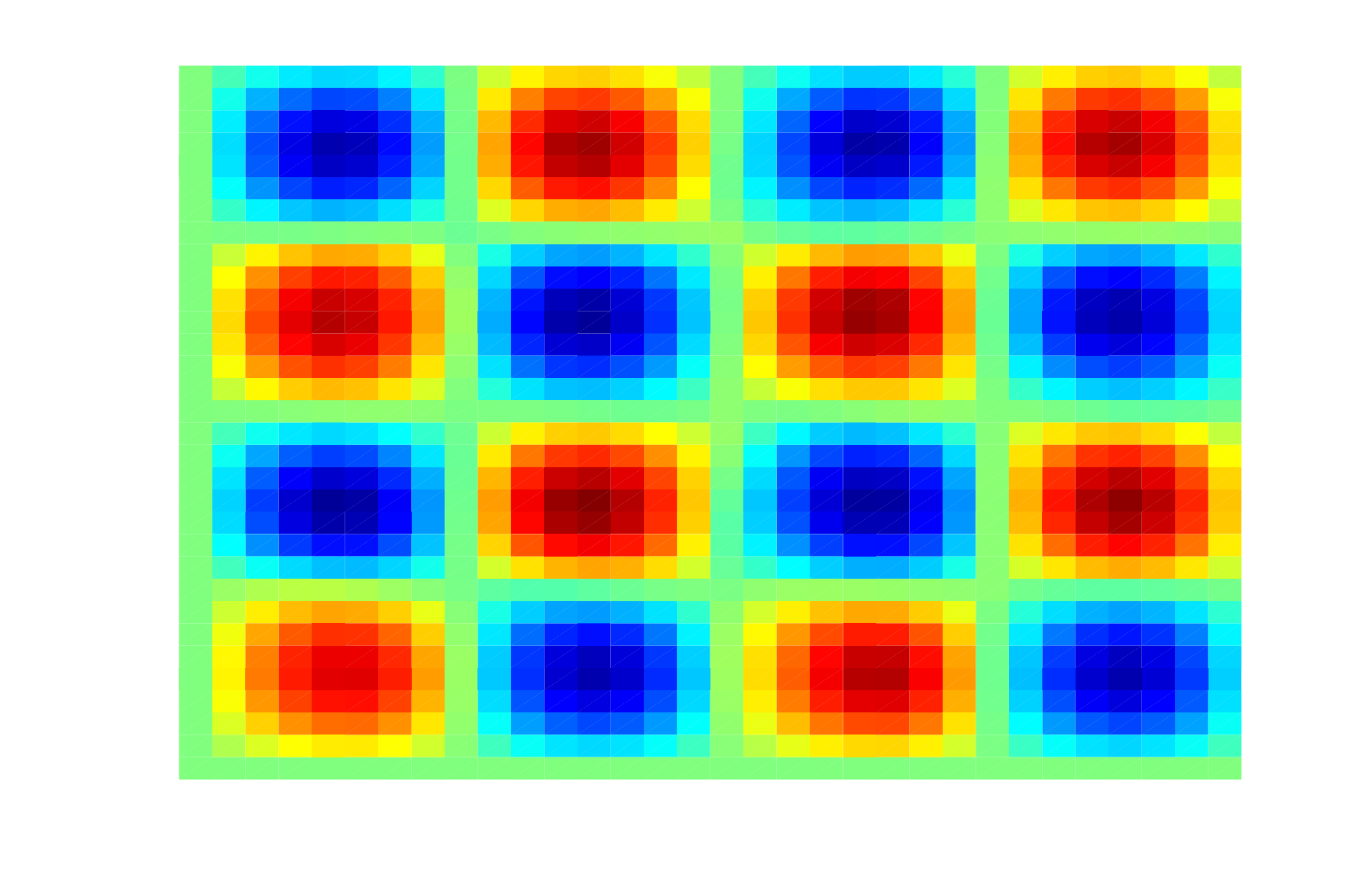}
        \captionsetup{labelformat=empty}
  \end{minipage}  
  \begin{minipage}[b]{0.5\linewidth}
  		\centering
  		\includegraphics[scale=.4]{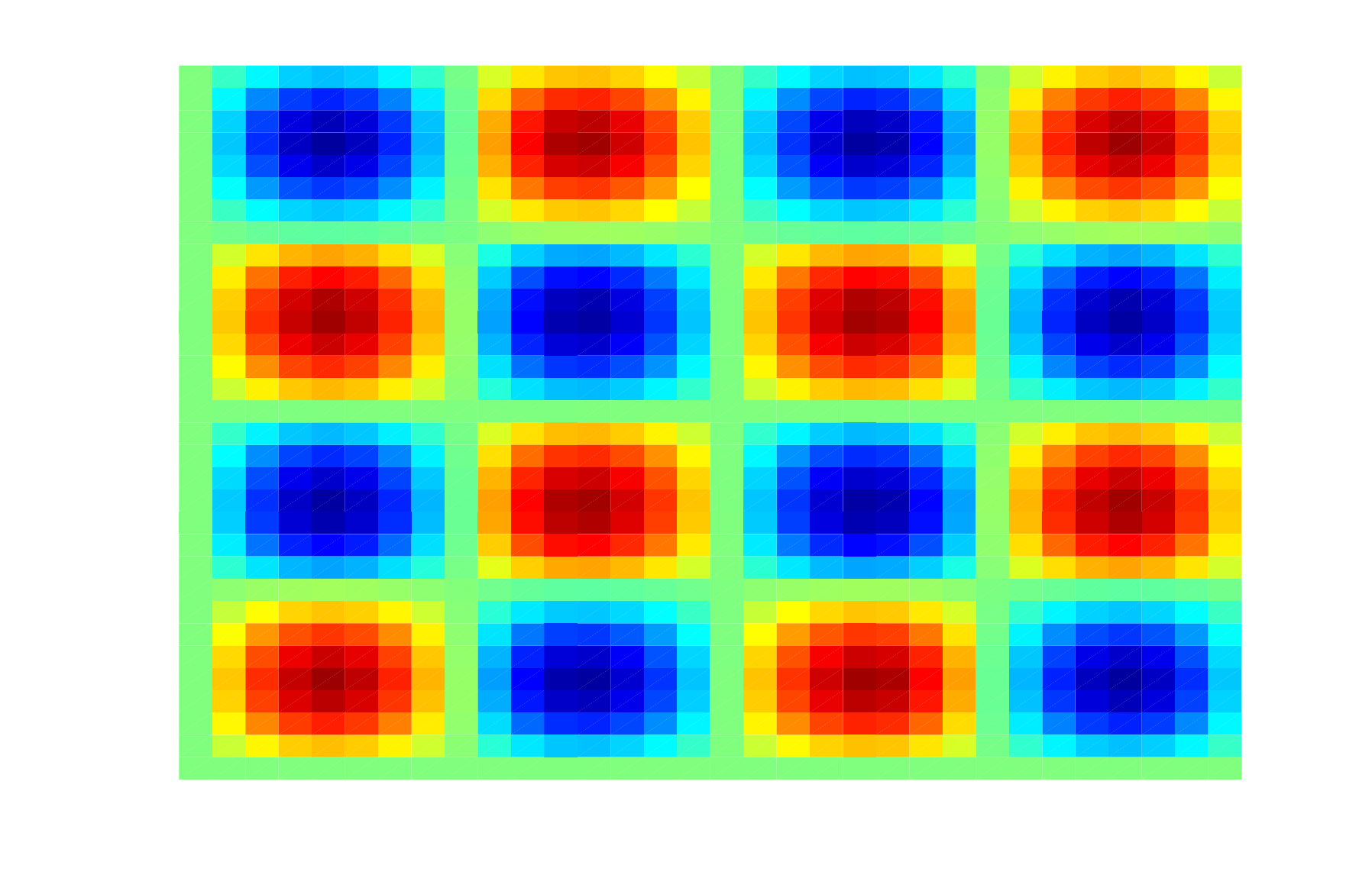}
        \captionsetup{labelformat=empty}
  \end{minipage} 
  \begin{minipage}[b]{0.5\linewidth}
		\centering
        \includegraphics[scale=.4]{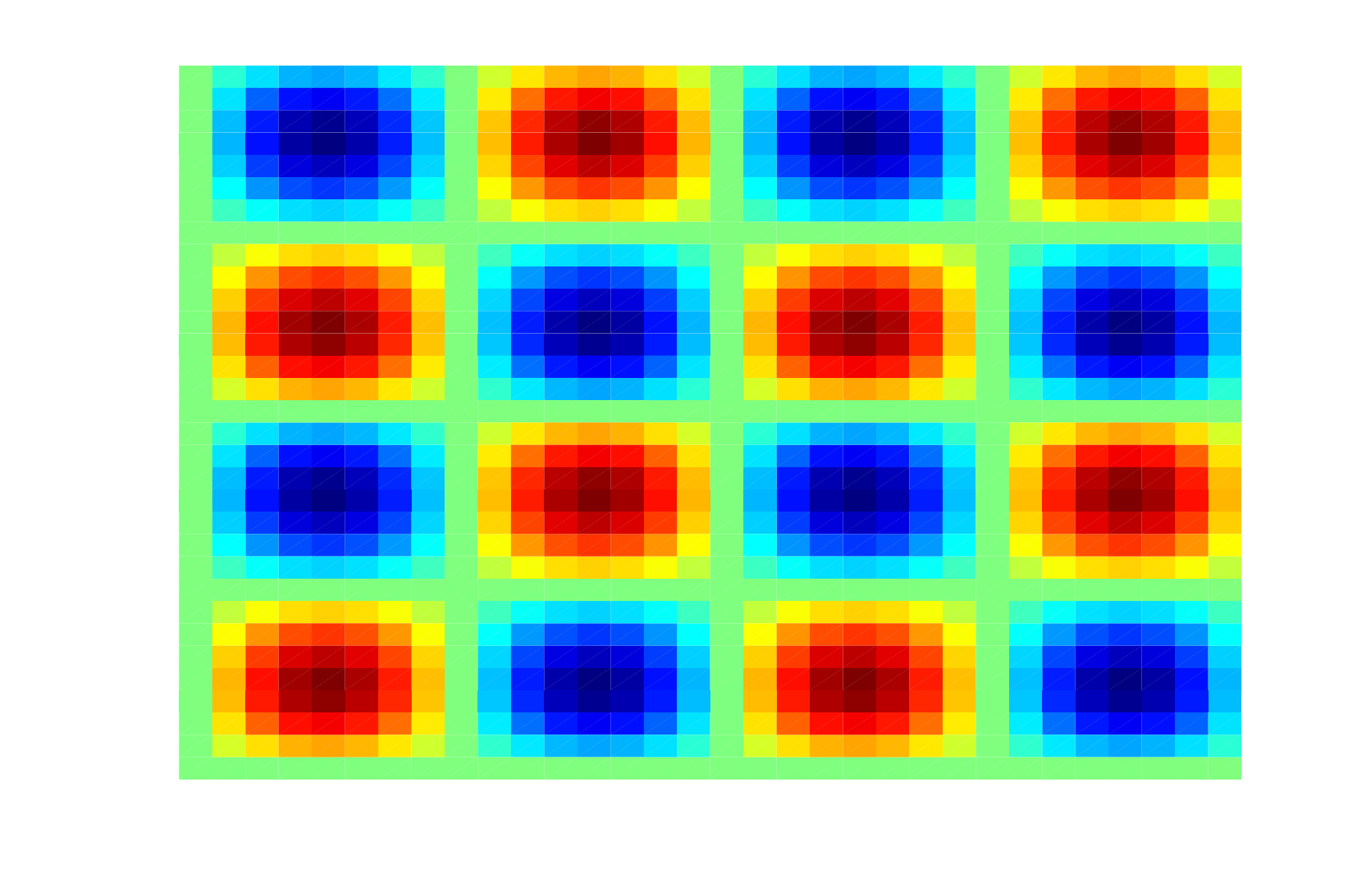}
        \captionsetup{labelformat=empty}
  \end{minipage}
  \caption{Level plots of $3$ approximations of exact solution $u_2$ in (\ref{Exemplo2}). (Up-Left 1-$N_N$, 0-$N_D$), (Up-Right 2-$N_N$, 0-$N_D$) and (Down-Left 3-$N_N$, 0-$N_D$) evolving to exact solution (Down-Right).}
\label{Smooth2}
\end{figure}

In the first example, including Dirichlet functions doesn't have much impact in the approximation, the error in the energy norm decreases faster by including more Newmann functions. While in the second example by including Dirichlet functions the energy norm  decreases faster than by including   Newman functions as we can see in Figure \ref{EnergyErrors1}\\

Now let us consider two different heterogeneous fields. The first is shown is 
composed by 3 chanels of high permeability in $64 \times 64$ fictitious geological 
mesh shown in Figure \ref{Canais}. The second heterogeneous medium to be consider 
will be the last $64 \times 64$ block of the geological permeability $SPE10$ porous
medium taken from \cite{SPEProject} (see Figure \ref{sp10}). This is a widely 
used heterogeneous porous medium for simulations (see for example \cite{TenhSPE})\\

\begin{figure}[h!]
\centering
\includegraphics[scale=.5]{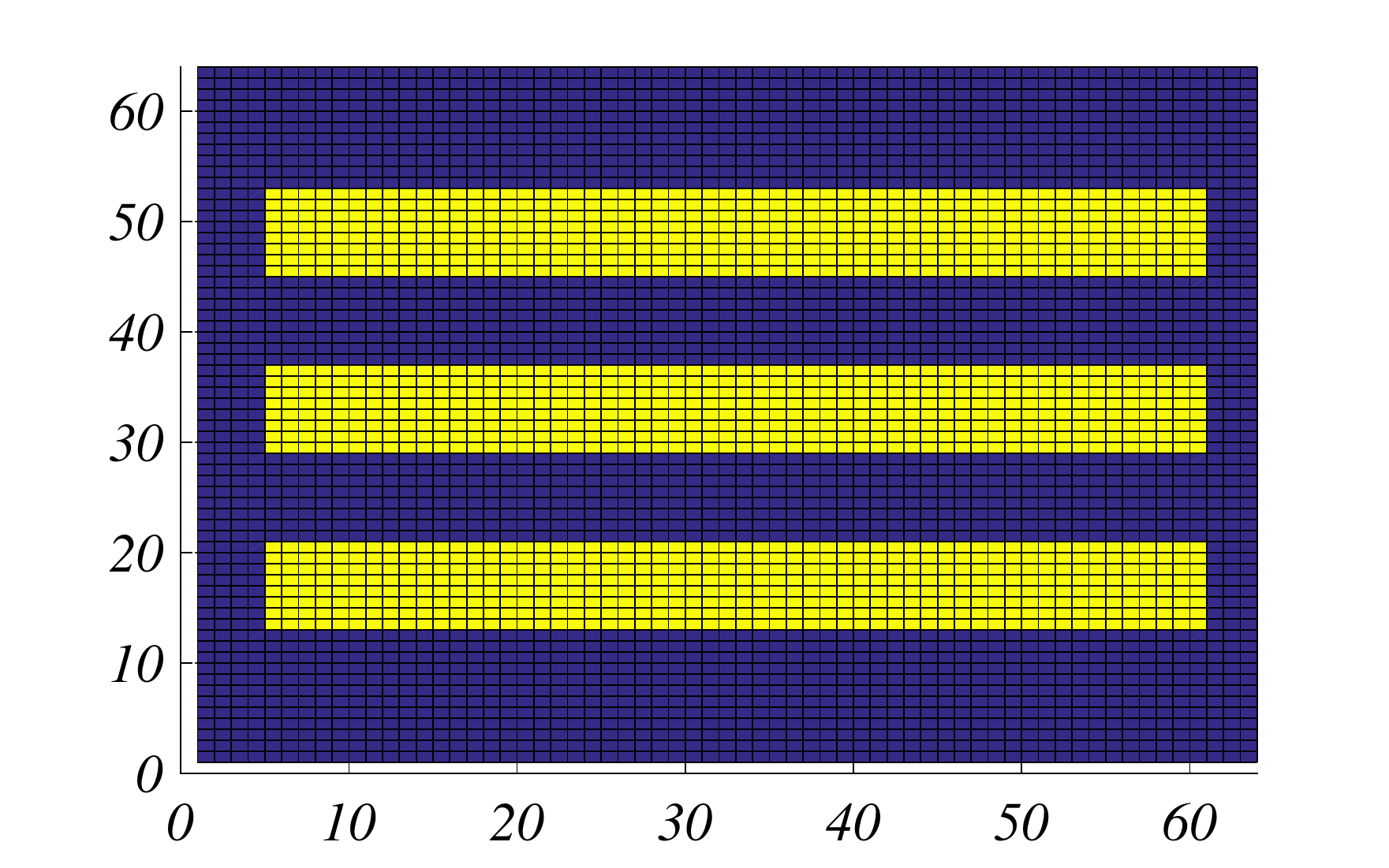} 
\caption{Heterogenous high permeability \textit{$3$-channels} medium in a $64 \times 64$ mesh.}
\label{Canais}
\end{figure}

\begin{figure}[h!]
\centering
\includegraphics[clip, trim={1.9cm 6.4cm 1.9cm 3cm}, scale=1.1]{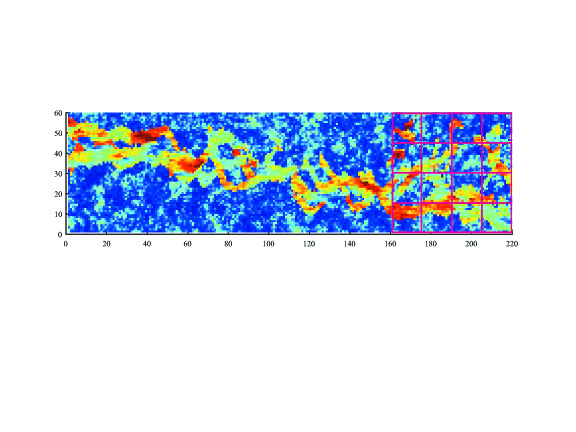}
\caption{A 2D layer of the $SPE10$ 2D porous medium sample from \cite{SPEProject}.}
\label{sp10}
\end{figure}  

Next we perform two numerical experiments to evaluate the performance of including Dirichlet basis when heterogeneity of the medium is present. \\

 \textbf{Experiment $3$}: Let us consider a source taking alternating values $1$ and $-1$ on the  $4 \times 4$ coarse mesh $\mathcal{T}^H$ of $\Omega$ 
 \begin{equation}
     f_3(x,y) = 
     \begin{cases}
        1 & if ~~ (x,y) \in K_{2j}\\
       -1 & if ~~ (x,y) \in K_{2j-1}
     \end{cases}
     \label{Chess_font}
 \end{equation}
 with $K_{j} \in \mathcal{T}^H$ and $j \in \{1,..,16\}$. The coarse mesh basis are computed on the fine $512 \times 512$ mesh $\mathcal{T}^h$ and approximated reference solution is computed in a $1024\times 1024$ finner mesh using classical finite element $\mathcal{Q}^1$. As we did in  Examples $1$ and $2$  we fix the number of Dirichlet coarse basis in the expansion $N_D$ from $1$ to $20$ and vary the number of Newmann coarse basis $N_N$ from $1$ to $40$. We use the  \textit{$3$-channels} heterogeneous medium shown in Figure \ref{Canais}\\

\textbf{Experiment $4$}: Exactly the same parameters in Experiment $3$ except for the medium which will be taken from  SPE10 as previously described and shown in Figure \ref{sp10}.\\

\begin{figure}[h!]
  \begin{minipage}[b]{0.5\linewidth}
  		\centering
  		\includegraphics[scale=.5]{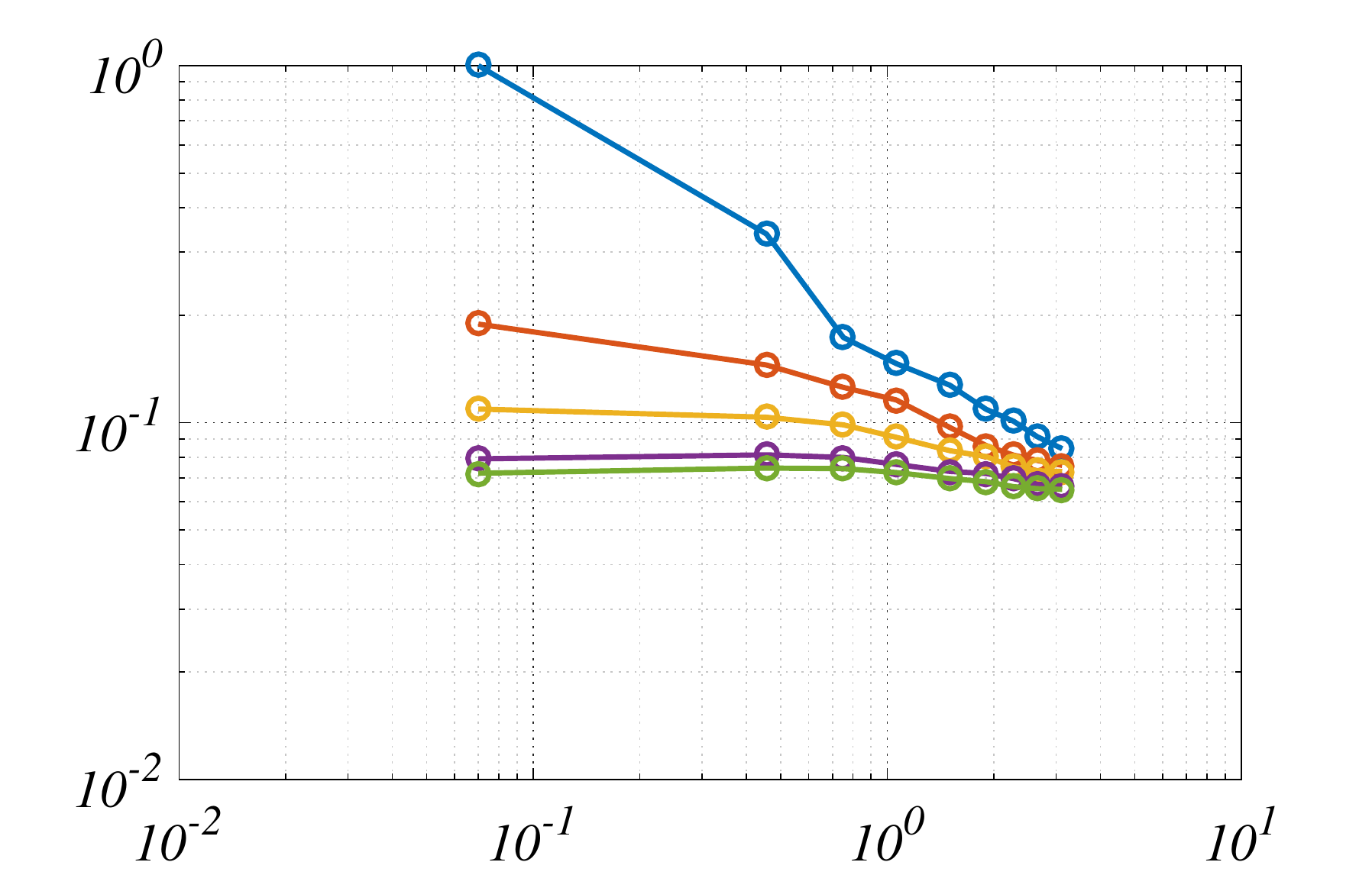}
  \end{minipage}  
  \begin{minipage}[b]{0.5\linewidth}
  		\centering
        \includegraphics[scale=.5]{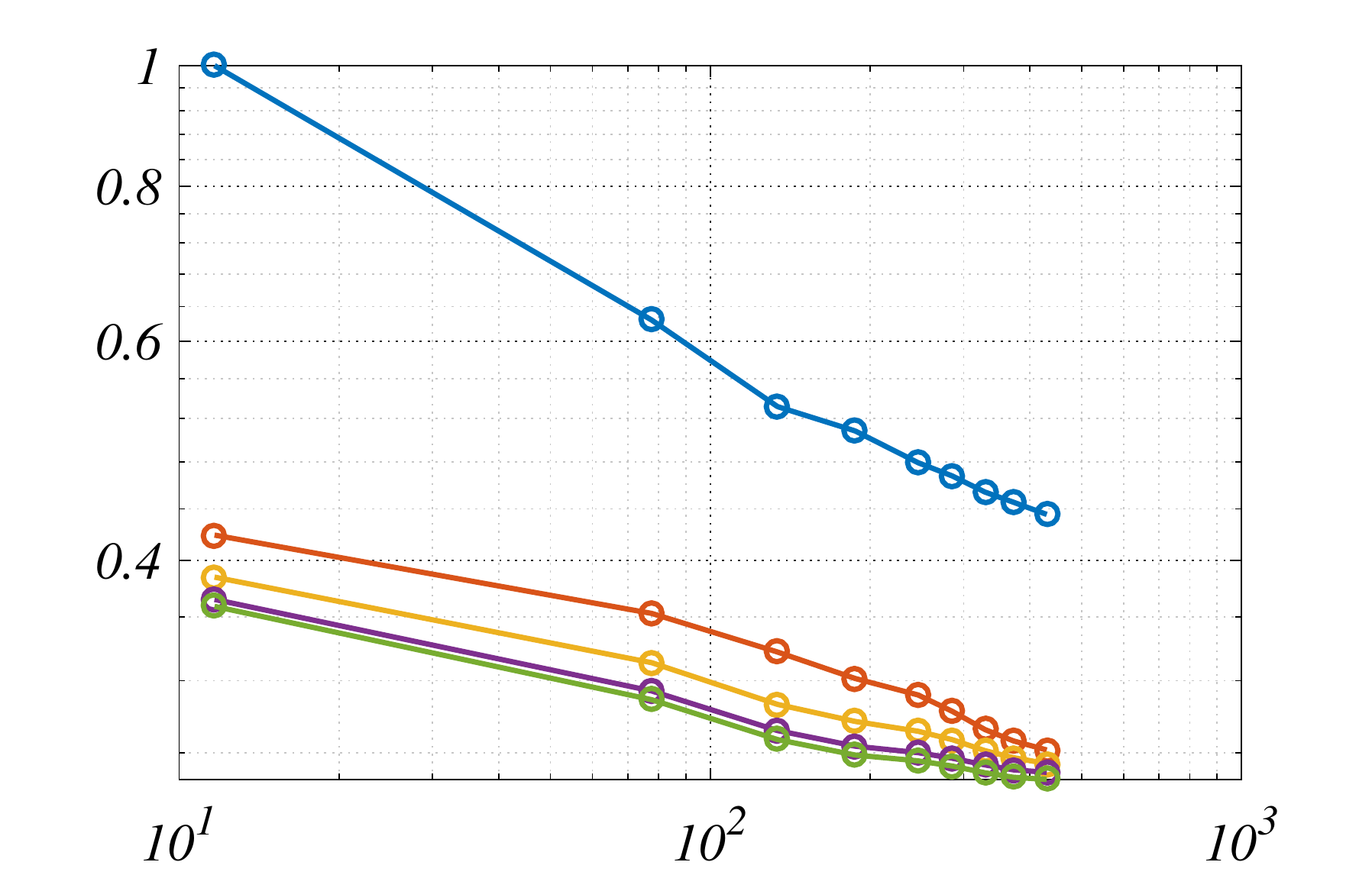}
  \end{minipage}  
  \caption{Log-Log plots of errors in the energy norm  Vs. first eigenvalue out of the expansion after applying GMsFEM to Experiments $3$ and $4$ using source term in (\ref{Chess_font}).(Left) Heterogenous $3-chanels$ medium (Figure \ref{Canais}). (Right) heterogeneous $SPE10$ medium (Figure \ref{sp10}). Each color correspond to a fixed $N_D = \{ 0,5,10,15,20\}$ and each circle correspond to $N_N = \{1,5,10,...,40\}$.} 
\end{figure}

\begin{figure}[h!]
\vspace{-1cm}
  \begin{minipage}[b]{0.5\linewidth}
  		\centering
  		\includegraphics[scale=.4]{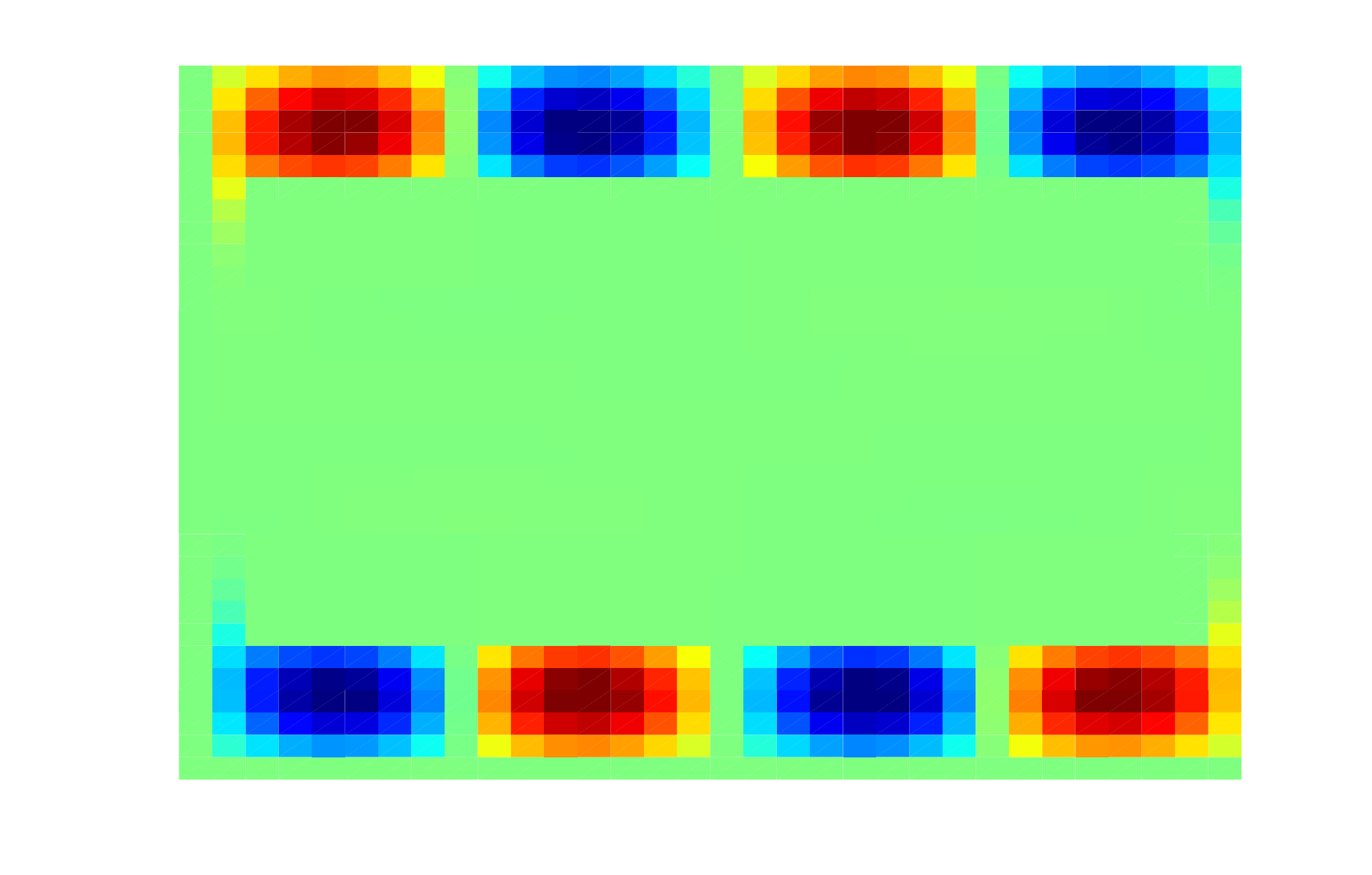}
        \captionsetup{labelformat=empty}
  \end{minipage}  
  \begin{minipage}[b]{0.5\linewidth}
  		\centering
        \includegraphics[scale=.4]{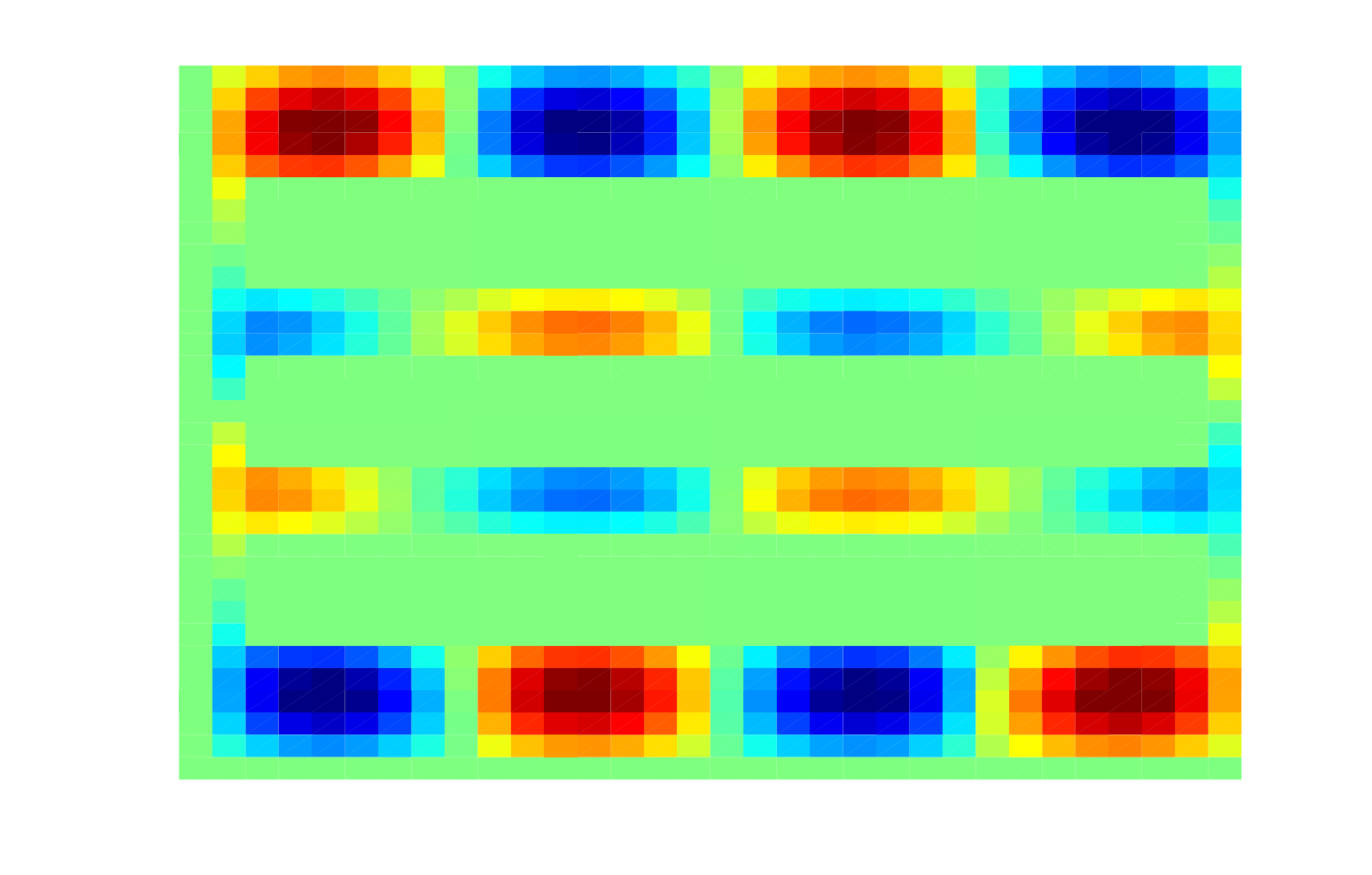}
        \captionsetup{labelformat=empty}
  \end{minipage}  
  \begin{minipage}[b]{0.5\linewidth}
  		\centering
  		\includegraphics[scale=.4]{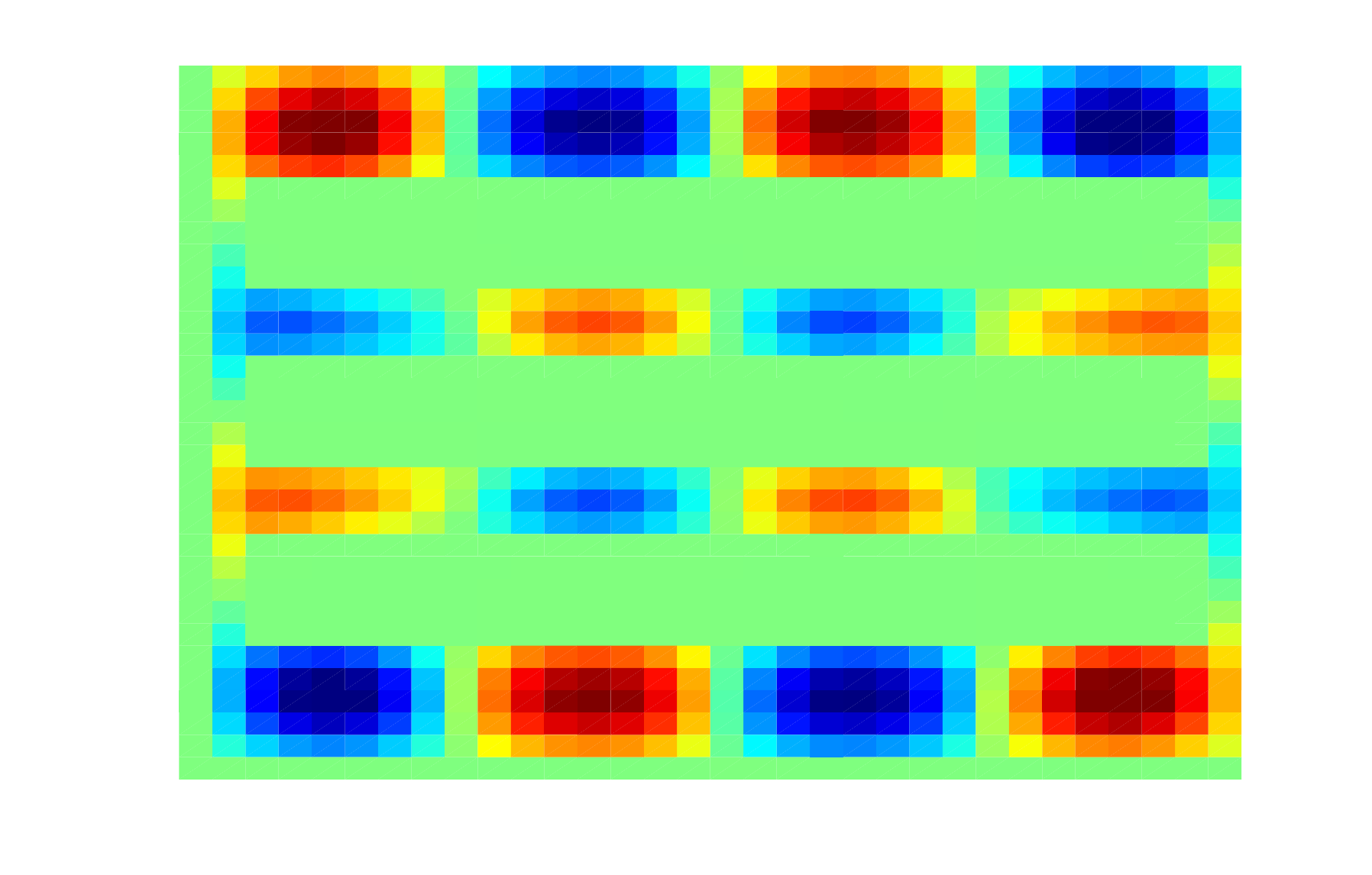}
        \captionsetup{labelformat=empty}
  \end{minipage} 
  \begin{minipage}[b]{0.5\linewidth}
		\centering
        \includegraphics[scale=.4]{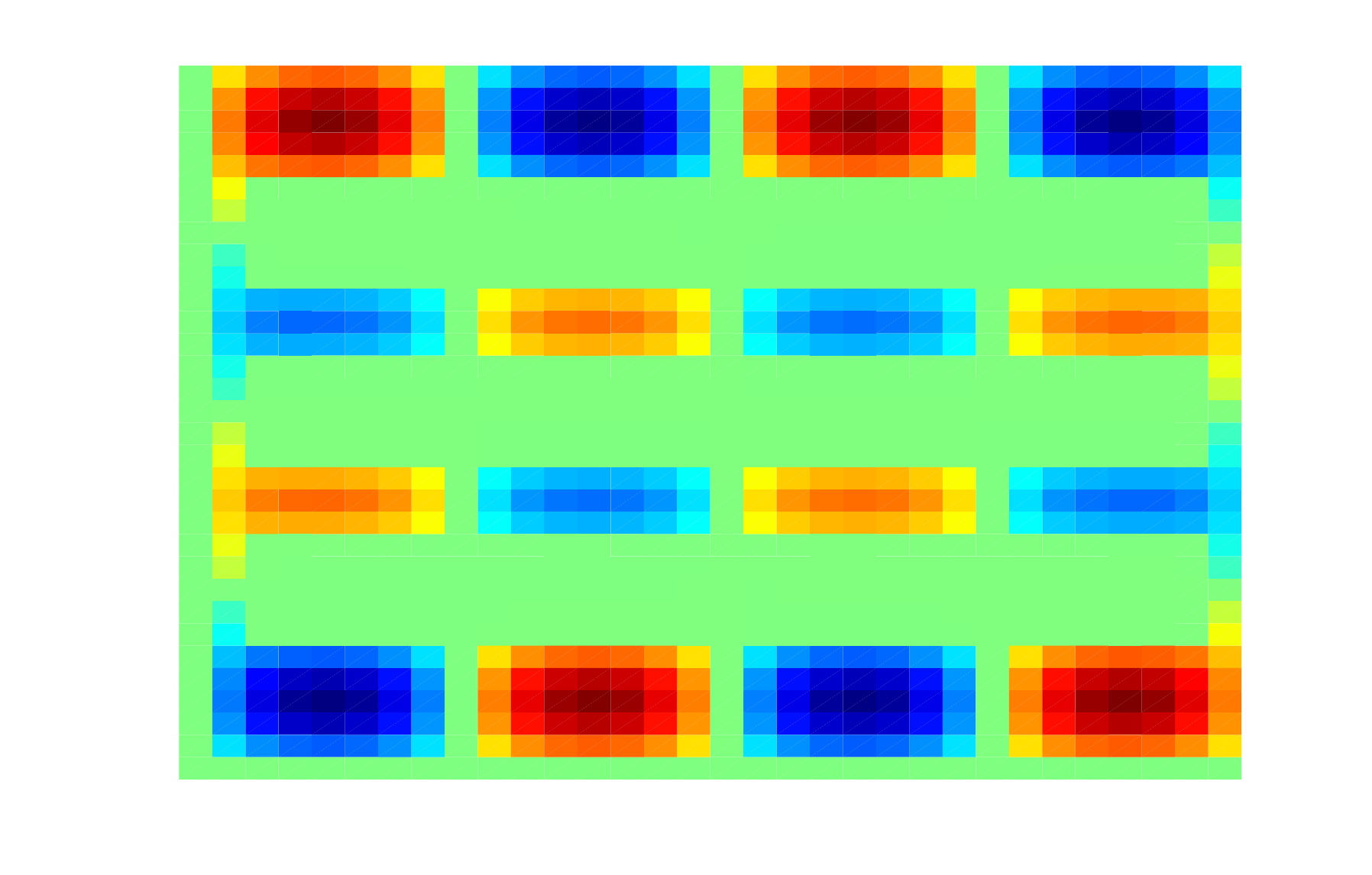}
        \captionsetup{labelformat=empty}
  \end{minipage}
  \caption{Level plots of $3$ approximations of solution of Experiment 
  $3$ (Up-Left 1-$N_N$, 0-$N_D$), (Up-Right 5-$N_N$, 0-$N_D$) and (Down-Left 10-$N_N$, 0-$N_D$) evolving to reference solution (Down-Right) computed in a finner mesh $1024 \times 1024$.}
\end{figure}

\begin{figure}[h!]
\vspace{-0.4cm}
  \begin{minipage}[b]{0.5\linewidth}
  		\centering
  		\includegraphics[scale=.4]{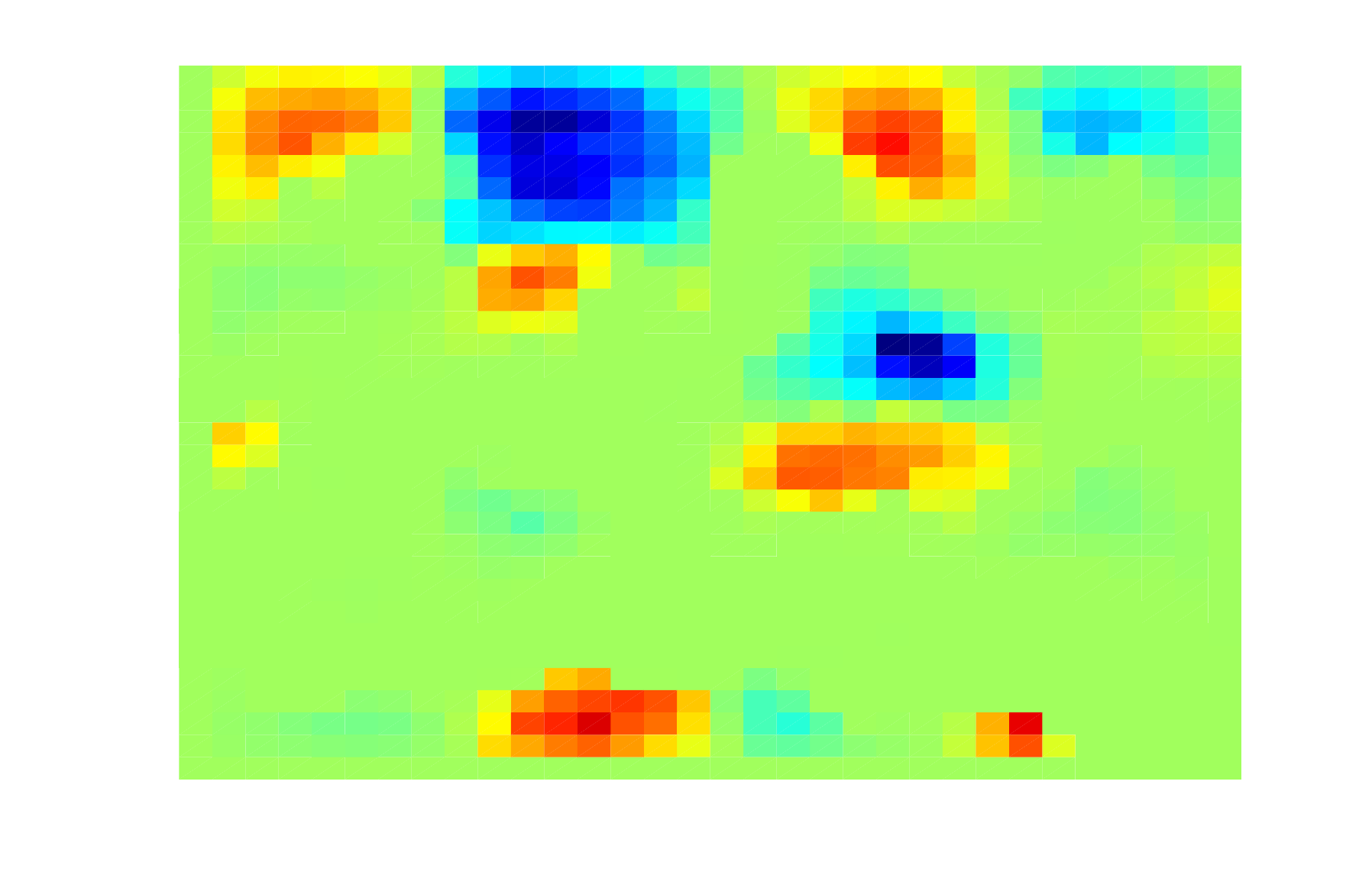}
        \captionsetup{labelformat=empty}
  \end{minipage}  
  \begin{minipage}[b]{0.5\linewidth}
  		\centering
        \includegraphics[scale=.4]{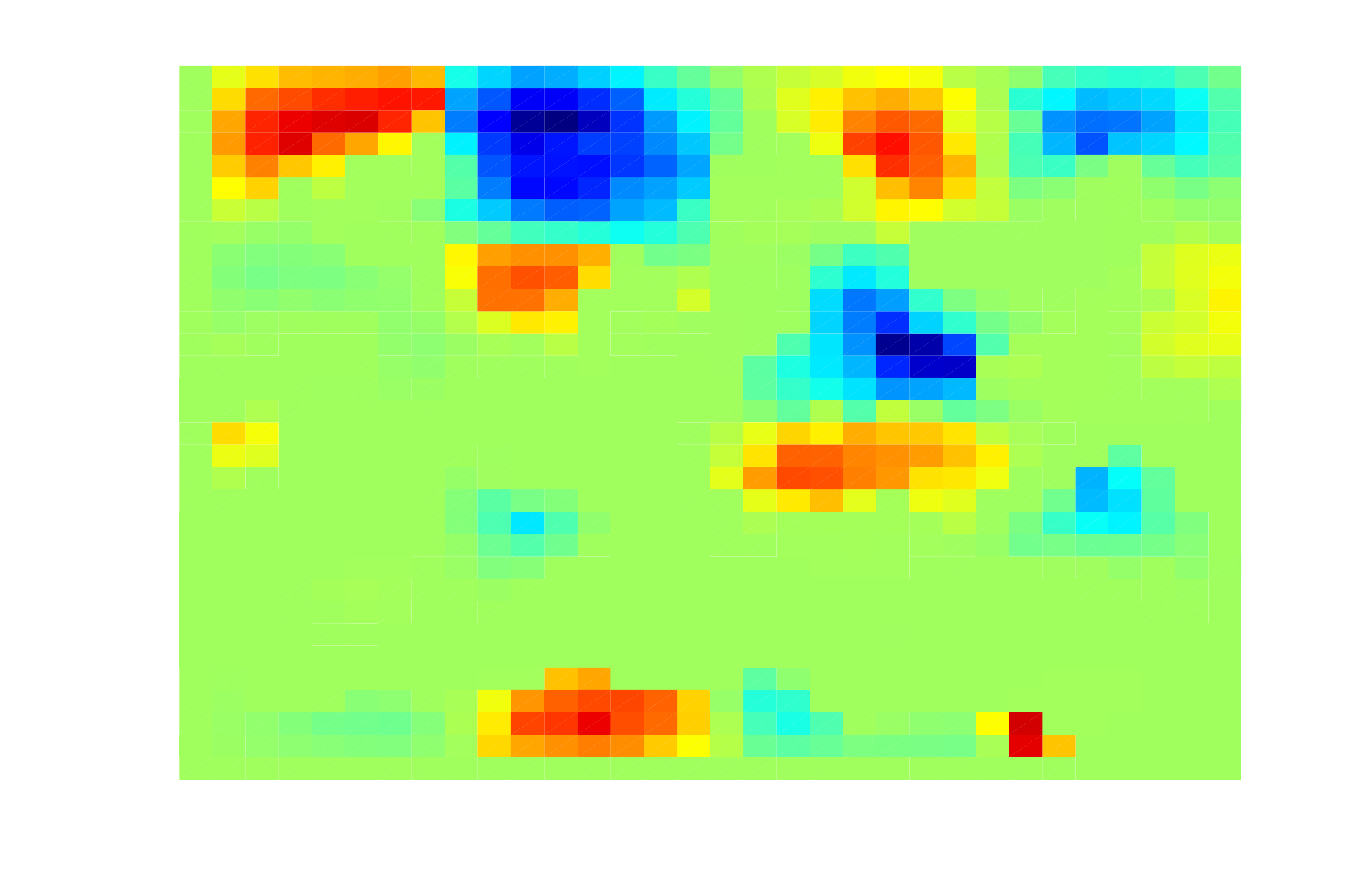}
        \captionsetup{labelformat=empty}
  \end{minipage}  
  \begin{minipage}[b]{0.5\linewidth}
  		\centering
  		\includegraphics[scale=.4]{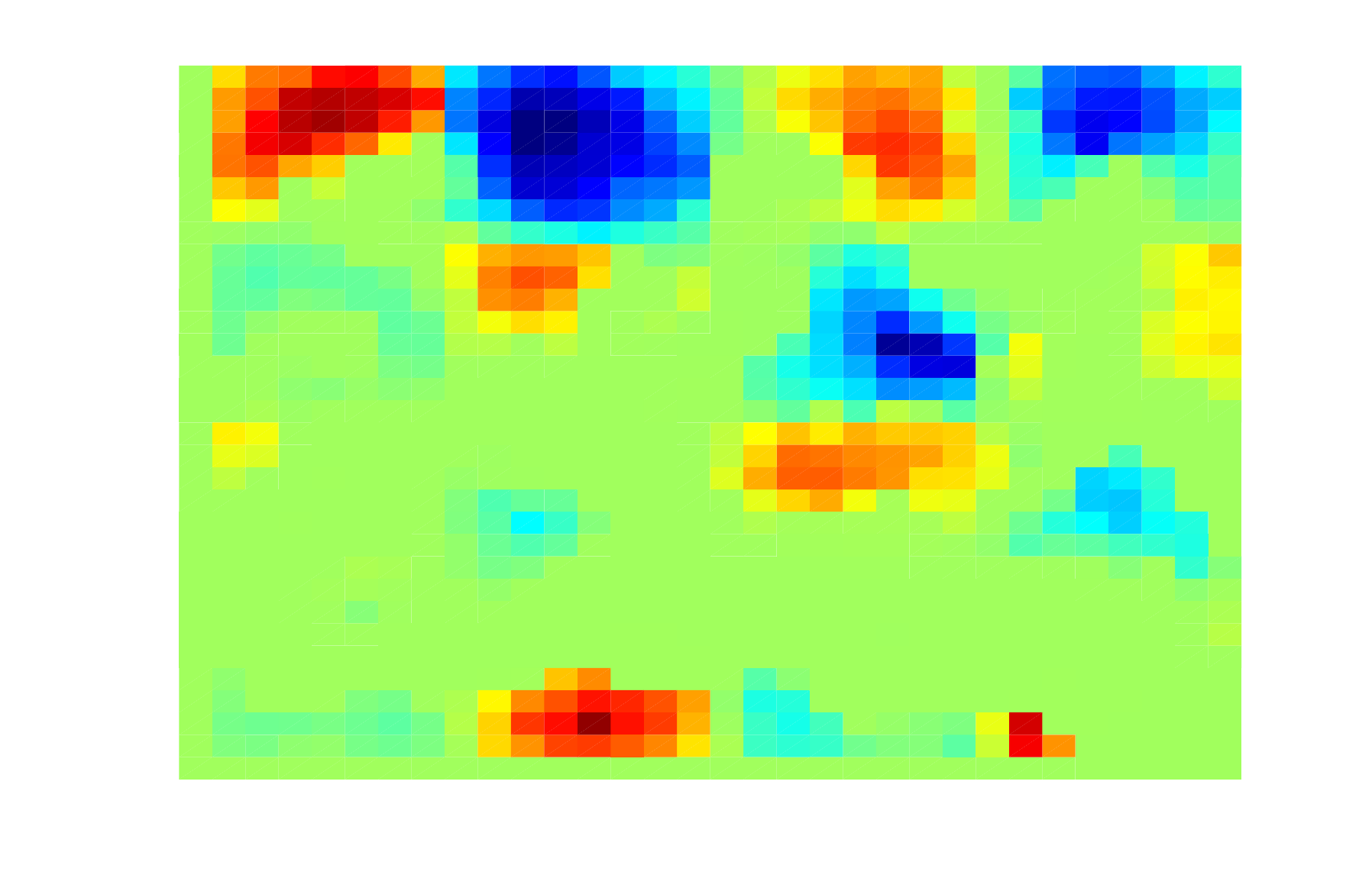}
        \captionsetup{labelformat=empty}
  \end{minipage} 
  \begin{minipage}[b]{0.5\linewidth}
		\centering
        \includegraphics[scale=.4]{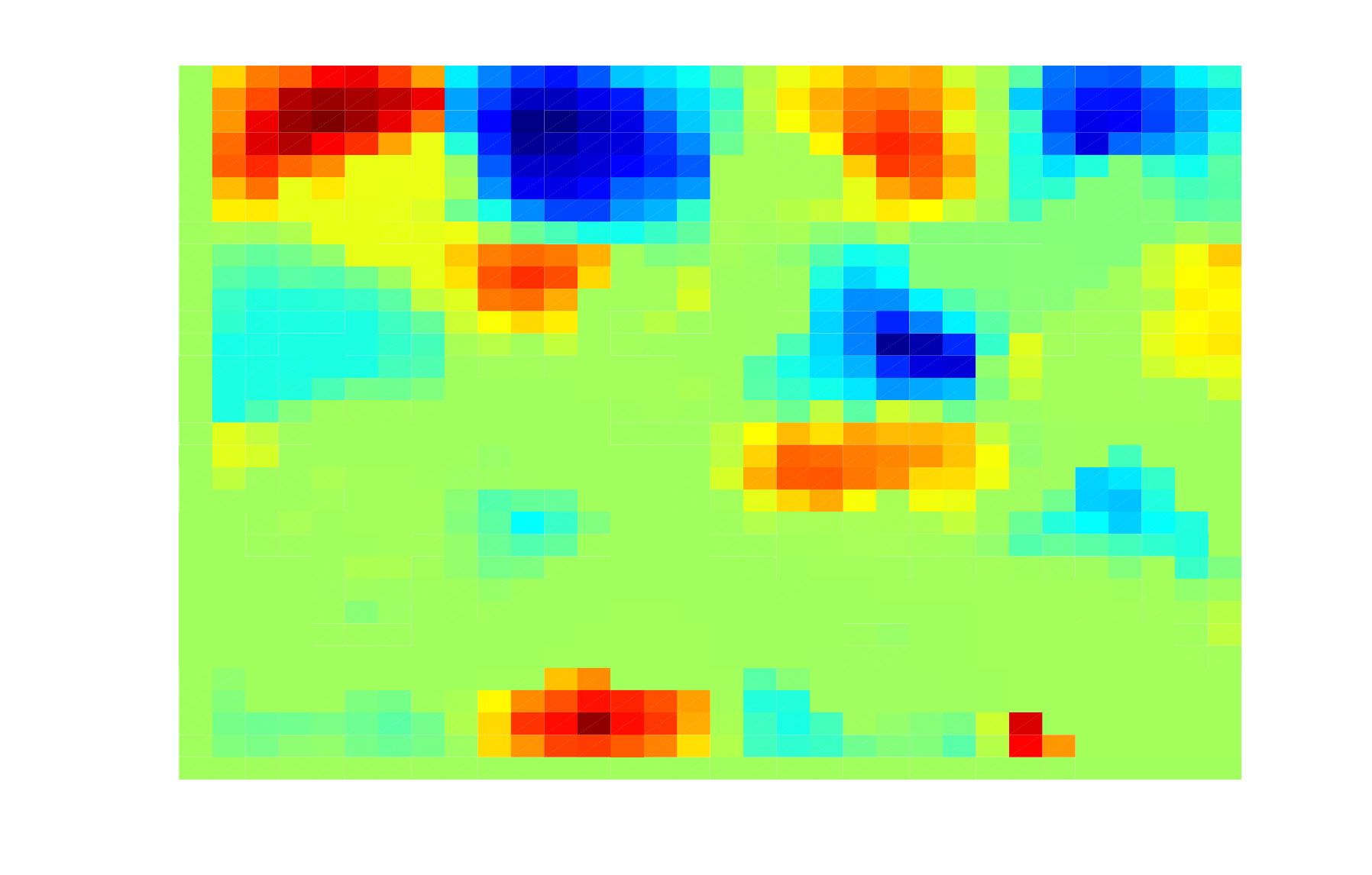}
        \captionsetup{labelformat=empty}
  \end{minipage}
  \caption{Level plots of $3$ approximations of  solution of Experiment $4$. (Up-Left 40-$N_N$, 0-$N_D$), (Up-Right 10-$N_N$, 5-$N_D$) and (Down-Left 40-$N_N$, 20-$N_D$) evolving to reference solution (Down-Right) computed in a finner mesh $1024 \times 1024$.}
\end{figure}

Both for Experiment $3$ and $4$ the approximated solution benefits from including Dirichlet basis in the expansion.

\section{Discussions and final comments}\label{sec:discussions}

In this paper, we obtained error estimates for GMsFEM approximation of 
high-contrast multiscale problems.  This construction uses local Neumann eigenvectors on neighborhoods and 
Dirichlet eigenvectors on elements 
to construct finite element basis function. The analysis is based on eigenfunction 
expansions and the norms used for 
the error estimates measure the decay of the expansion of the solution in terms of 
local eigenfunctions.  The norms in the interpolation error estimates can be  
bounded by the $L^2$ norm of a rescaling of the forcing term. For the analysis we assume that the solution can be approximated by a sum of two functions, one with zero flux across coarse blocks boundaries, and the other with zero value on coarse blocks boundaries. This assumption is easily verified for classical regular problems. The introduction of this assumption allowed us to extend and simplify the convergence analysis presented in \cite{egw10}.
The error estimates derived here can be applied to several situations that are under 
current research. For instance, it is possible to obtain error estimates for general 
bilinear forms using the analysis presented here combined 
with the construction of coarse spaces in \cite{EGLW12}.

\newpage
\section*{Notation}

In general we use the letters $\mu$ and $\phi$ to refer to Dirichlet eigenvalues and eigenvectors; and we use letters $\lambda$ and $\psi$ to refer to Newmann eigenvalues and eigenvectors. Letter $\Phi$ is reserved to generalized basis. 

\begin{table}[h!]\caption{Table of Notation}
\begin{center}
\begin{tabular}{r c p{10cm} }
\toprule
$\Omega$        &  pag.4  & Definition domain of the elliptic problem \\
$H^1(\Omega)$   &  pag.9 & Sobolev space of functions with continuous first derivatives on $\Omega$ \\
$H_0^1(\Omega)$ &  pag.4  & Sobolev space of functions with continuous first derivatives on $\Omega$ and vanishing on $\partial\Omega$\\
$w_i$           &  pag.5  & Neighborhood of the node $y_i$ of a triangulation.\\
$w^K$           &  pag.5 & Union of all neighborhood containing the element $K$\\
$\mu^K_l$       &  pag.10 &  $l-th$ Dirichlet eigenvalue of the local eigenvalue problem on $K$\\
$\phi^K_l$ &  pag.10  &  $l-th$ Dirichlet eigenfunction of the local eigenvalue problem on $K$.\\
$\lambda^{w_i}_l$ & pag.12   &  $l-th$ Newmann eigenvalue of the local eigenvalue problem on $w_i$.\\
$\psi^{w_i}_l$ &  pag.12  &  $l-th$ Newmann eigenfunction of the local eigenvalue problem on $w_i$.\\
$\Phi_i$       &  pag.12 & Generalased basis function associated to the neighborhood $w_i$ \\
$\mathcal{A}$  & pag.7  & Elliptic operator on $\Omega$.\\
$\mathcal{A}^K$ & pag.11 & Elliptic operator on the coarse block $K$.\\
$\mathcal{A}^{w_i}$ & pag.13 & elliptic operator defined on the neighborhood $w_i$.\\
$\mathcal{J}_L$ & pag.9 & Dirichlet projection operator truncated at $L$.\\
$\mathcal{J}_L^K$ & pag.13 & Dirichlet projection operator on element $K$ truncated at $L$.\\
$I^{\omega_i}_{L}$ & pag.18 & Newmann projection operator in on neighborhood $\omega_i$ truncated at $L$.\\
$I_{N}$ & pag.19 & Generalized basis coarse interpolation operator\\
$J_{D}$ & pag.19 & Dirichlet basis coarse interpolation operator\\
$ \widetilde{V}(\omega_i)$ & pag.12 & Space of functions in $H^1(w_i)$ which are zero on $\partial w_i \cap \partial\Omega$.\\
$|||v|||_{s;{D}}$ & pag.6 & Norms based on the eigenvalue expansion decay on an set $D$.
\end{tabular}
\end{center}
\end{table}

\section*{Acknowledgments}

Eduardo Abreu thanks the FAPESP for support under grant 2016/23374-1.
Juan Galvis  wants to thank KAUST hospitality where part of this work was developed and also the discussion on coarse space approximations properties and related topics with several colleagues, among them, Joerg Willems, Marcus Sarkis, Raytcho Lazarov,  Jhonny Guzm\'an, 
Chia-Chieh Chu, Florian Maris and Yalchin Efendiev.

\section*{References}
\bibliographystyle{elsarticle-num}
\bibliography{references}

\begin{thebibliography}{10}
\expandafter\ifx\csname url\endcsname\relax
  \def\url#1{\texttt{#1}}\fi
\expandafter\ifx\csname urlprefix\endcsname\relax\def\urlprefix{URL }\fi
\expandafter\ifx\csname href\endcsname\relax
  \def\href#1#2{#2} \def\path#1{#1}\fi

\bibitem{hw97}
T.~Hou, X.~Wu, A multiscale finite element method for elliptic problems in
  composite materials and porous media, J. Comput. Phys. 134 (1997) 169--189.

\bibitem{aarnes}
J.~Aarnes, T.~Hou, Multiscale domain decomposition methods for elliptic
  problems with high aspect ratios, Acta Math. Appl. Sin. Engl. Ser. 18 (2002)
  63--76.

\bibitem{aej07}
J.~Aarnes, Y.~Efendiev, L.~Jiang, Analysis of multiscale finite element methods
  using global information for two-phase flow simulations, SIAM J. Multiscale
  Modeling and Simulation 7 (2008) 2177--2193.

\bibitem{bo09}
L.~Berlyand, H.~Owhadi, A new approach to homogenization with arbitrary rough
  high contrast coeffcients for scalar and vectorial problems, Submitted.

\bibitem{AKL}
J.~Aarnes, S.~Krogstad, K.-A. Lie, A hierarchical multiscale method for
  two-phase flow based upon mixed finite elements and nonuniform grids, SIAM J.
  Multiscale Modeling and Simulation 5~(2) (2006) 337--363.

\bibitem{arbogast02}
T.~Arbogast, Implementation of a locally conservative numerical subgrid
  upscaling scheme for two-phase {D}arcy flow, Comput. Geosci 6 (2002)
  453--481.

\bibitem{apwy07}
T.~Arbogast, G.~Pencheva, M.~Wheeler, I.~Yotov, A multiscale mortar mixed
  finite element method, SIAM J. Multiscale Modeling and Simulation 6~(1)
  (2007) 319--346.

\bibitem{cdgw03}
Y.~Chen, L.~Durlofsky, M.~Gerritsen, X.~Wen, A coupled local-global upscaling
  approach for simulating flow in highly heterogeneous formations, Advances in
  Water Resources 26 (2003) 1041--1060.

\bibitem{eghe05}
Y.~Efendiev, V.~Ginting, T.~Hou, R.~Ewing, Accurate multiscale finite element
  methods for two-phase flow simulations, Journal of Computational Physics 220
  (2006) 155--174.

\bibitem{jennylt03}
P.~Jenny, S.~Lee, H.~Tchelepi, Multi-scale finite volume method for elliptic
  problems in subsurface flow simulation, J. Comput. Phys. 187 (2003) 47--67.

\bibitem{cgh09}
C.~Chu, I.~Graham, T.~Hou, A new multiscale finite element methods for
  high-contrast elliptic interface problem, Mathematics of Computation 79
  (2010) 1915--1955.

\bibitem{hughes98}
T.~Hughes, G.~Feijoo, L.~Mazzei, J.~Quincy, The variational multiscale method -
  a paradigm for computational mechanics, Comput. Methods Appl. Mech. Engrg.
  166 (1998) 3--24.

\bibitem{eh09}
Y.~Efendiev, T.~Hou, {Multiscale Finite Element Methods: Theory and
  Applications}, Vol.~4 of Surveys and Tutorials in the Applied Mathematical
  Sciences, Springer, New York, 2009.

\bibitem{ge09_1}
J.~Galvis, Y.~Efendiev, Domain decomposition preconditioners for multiscale
  flows in high contrast media, SIAM J. Multiscale Modeling and Simulation 8
  (2010) 1461--1483.

\bibitem{ge09_1reduceddim}
J.~Galvis, Y.~Efendiev, Domain decomposition preconditioners for multiscale
  flows in high contrast media. reduced dimension coarse spaces, SIAM J.
  Multiscale Modeling and Simulation 8 (2010) 1621--1644.

\bibitem{sarkisguzman}
E.~Burman, J.~Guzm\'{a}n, M.~A. S\'{a}nchez, M.~Sarkis,
  \href{https://doi.org/10.1093/imanum/drx017}{Robust flux error estimation of
  an unfitted {N}itsche method for high-contrast interface problems}, IMA J.
  Numer. Anal. 38~(2) (2018) 646--668.
\newblock \href {http://dx.doi.org/10.1093/imanum/drx017}
  {\path{doi:10.1093/imanum/drx017}}.
\newline\urlprefix\url{https://doi.org/10.1093/imanum/drx017}

\bibitem{sarkisburman}
E.~Burman, J.~Guzm\'{a}n, M.~A. S\'{a}nchez, M.~Sarkis,
  \href{https://doi.org/10.1093/imanum/drx017}{Robust flux error estimation of
  an unfitted {N}itsche method for high-contrast interface problems}, IMA J.
  Numer. Anal. 38~(2) (2018) 646--668.
\newblock \href {http://dx.doi.org/10.1093/imanum/drx017}
  {\path{doi:10.1093/imanum/drx017}}.
\newline\urlprefix\url{https://doi.org/10.1093/imanum/drx017}

\bibitem{tat}
E.~T. Chung, Y.~Efendiev, W.~T. Leung,
  \href{https://doi.org/10.1137/140986189}{An adaptive generalized multiscale
  discontinuous {G}alerkin method for high-contrast flow problems}, Multiscale
  Model. Simul. 16~(3) (2018) 1227--1257.
\newblock \href {http://dx.doi.org/10.1137/140986189}
  {\path{doi:10.1137/140986189}}.
\newline\urlprefix\url{https://doi.org/10.1137/140986189}

\bibitem{ehg04}
Y.~Efendiev, T.~Hou, V.~Ginting, Multiscale finite element methods for
  nonlinear problems and their applications, Comm. Math. Sci. 2 (2004)
  553--589.

\bibitem{ehw99}
Y.~Efendiev, T.~Hou, X.~Wu, Convergence of a nonconforming multiscale finite
  element method, SIAM J. Numer. Anal. 37 (2000) 888--910.

\bibitem{eglw11}
Y.~Efendiev, J.~Galvis, R.~Lazarov, J.~Willems, Robust domain decomposition
  preconditioners for abstract symmetric positive definite bilinear forms,
  ESIAM : M2AN 46 (2012) 1175--1199.

\bibitem{egw10}
Y.~Efendiev, J.~Galvis, X.~Wu, Multiscale finite element methods for
  high-contrast problems using local spectral basis functions, Journal of
  Computational Physics 230 (2011) 937--955.

\bibitem{CEG}
V.~M. Calo, Y.~Efendiev, J.~Galvis,
  \href{http://dx.doi.org/10.1142/S0218202513500565}{Asymptotic expansions for
  high-contrast elliptic equations}, Math. Models Methods Appl. Sci. 24~(3)
  (2014) 465--494.
\newblock \href {http://dx.doi.org/10.1142/S0218202513500565}
  {\path{doi:10.1142/S0218202513500565}}.
\newline\urlprefix\url{http://dx.doi.org/10.1142/S0218202513500565}

\bibitem{EGG_MultiscaleMOR}
Y.~Efendiev, J.~Galvis, E.~Gildin, Local-global multiscale model reduction for
  flows in highly heterogeneous media, Submitted.

\bibitem{Review}
Y.~Efendiev, J.~Galvis, Coarse-grid multiscale model reduction techniques for
  flows in heterogeneous media and applications, Chapter of Numerical Analysis
  of Multiscale Problems, Lecture Notes in Computational Science and
  Engineering, Vol. 83.  97--125.

\bibitem{egh12}
Y.~Efendiev, J.~Galvis, T.~Hou, Generalized multiscale finite element methods,
  Journal of Computational Physics 251 (2013) 116--135.

\bibitem{egt11}
Y.~Efendiev, J.~Galvis, F.~Thomines, A systematic coarse-scale model reduction
  technique for parameter-dependent flows in highly heterogeneous media and its
  applications, Multiscale Model. Simul. 10 (2012) 1317--1343.

\bibitem{EG09}
Y.~Efendiev, J.~Galvis, Domain decomposition preconditioner for multiscale
  high-contrast problems, in: Proceedings of DD19, 2009.

\bibitem{EfendievGLWESAIM12}
Y.~Efendiev, J.~Galvis, R.~Lazarov, J.~Willems,
  \href{http://dx.doi.org/10.1051/m2an/2011073}{Robust domain decomposition
  preconditioners for abstract symmetric positive definite bilinear forms},
  ESAIM Math. Model. Numer. Anal. 46~(5) (2012) 1175--1199.
\newblock \href {http://dx.doi.org/10.1051/m2an/2011073}
  {\path{doi:10.1051/m2an/2011073}}.
\newline\urlprefix\url{http://dx.doi.org/10.1051/m2an/2011073}

\bibitem{Babuska}
I.~Babu{\v{s}}ka, V.~Nistor, N.~Tarfulea, Generalized finite element method for
  second-order elliptic operators with {D}irichlet boundary conditions, J.
  Comput. Appl. Math. 218 (2008) 175--183.

\bibitem{Babuska2}
K.~C. T.~Strouboulis, I.~Babu{\v{s}}ka, The design and analysis of the
  generalized finite element method, Comput. Methods Appl. Mech. Engrg. 181
  (2000) 43--69.

\bibitem{SPEProject}
S.~of~petrleum ingeneers, Spe comparative solution project,
  https://www.spe.org/web/csp/.

\bibitem{TenhSPE}
M.~Christie, M.~Blunt, \href{https://doi.org/10.2118/72469-PA}{Tenth spe
  comparative solution project: A comparison of upscaling techniques}, Society
  of Petroleum Engineers\href {http://dx.doi.org/10.2118/72469-PA}
  {\path{doi:10.2118/72469-PA}}.
\newline\urlprefix\url{https://doi.org/10.2118/72469-PA}

\bibitem{EGLW12}
Y.~Efendiev, J.~Galvis, R.~Lazarov, J.~Willems, Robust domain decomposition
  preconditioners for abstract symmetric positive definite bilinear forms,
  ESAIM: Mathematical Modelling and Numerical Analysis (2012) 1175--1199.

\end{thebibliography}

\end{document}